\documentclass[11pt, reqno]{article}
\usepackage{amssymb, amsmath, amsthm, amsrefs}
\usepackage{latexsym, comment, enumitem, color, bbold, esint}

\def\F{{\mathcal F}} 
\def\N{{\mathbb N}} 
\def\R{{\mathbb R}} 
\def\T{{\mathbb T}} 

\def\Z{{\mathbb Z}} 

\def\blue{\textcolor{blue}}

\newtheorem{theorem}{Theorem}[section]
\newtheorem{assumption}[theorem]{Assumption}
\newtheorem{corollary}{Corollary}[theorem]
\newtheorem{lemma}[theorem]{Lemma}
\newtheorem{definition}[theorem]{Definition}
\newtheorem{proposition}[theorem]{Proposition}
\newtheorem{remark}[theorem]{Remark}

\def\ds{\displaystyle}

\def\grad{\nabla}
\def\E{\mathbf{E}}

\def\P{{\mathbb P}} 

\def\curl{{\mathrm curl}}
\def\Div{{\mathrm div}}

\def\supp{\mathrm{supp}}

\title{Necessary and Sufficient Conditions for Kolmogorov's Flux Laws on $\T^2$ and $\T^3$.}
\author{Ethan Dudley\footnotemark[2]}

\begin{document}
\maketitle
\renewcommand{\thefootnote}{\fnsymbol{footnote}}
\footnotetext[2]{ Mathematics Department, University of Maryland College Park, Maryland, United States}
\begin{abstract}
    Necessary and sufficient conditions for the third order Kolmogorov universal scaling flux laws are derived for the stochastically forced incompressible Navier-Stokes equations on the torus in 2d and 3d. This paper rigorously generalizes the result of \cite{bedrossian2019sufficient} to functions which are heavy-tailed in Fourier space or have local finite time singularities in the inviscid limit. In other words we have rigorously derived the well known physical relationship the direct cascade is a local process and is non-trivial if and only if energy moves toward the small scales or singularities have occurred. Similarly, an inverse cascade occurs if and only if energy moves towards the $k = 0$ Fourier mode in the invisicid limit. 
\end{abstract}
\tableofcontents

\section{Introduction}\label{sec: intro}
In this paper, we consider the stochastically forced incompressible Navier-Stokes equations on the Torus $\T_\lambda^d = [0,\lambda)^d$ for $d = 2,3$:
\begin{align}\label{eq: incompressible NSE}
    du^\nu + (u^\nu \cdot \grad u^\nu + \grad p^\nu)\;dt &= \nu \Delta u^\nu\;dt + \sum_{j=1}^\infty f_j\;dW_t^j\\
    \grad \cdot u^\nu &= 0\nonumber
\end{align}
Here $\nu$ is the kinematic viscosity and represents the inverse Reynolds number, $u^\nu$ is the viscous velocity, and $p^\nu$ is the associated pressure. The forcing is assumed to be white-in-time and coloured-in-space Gaussian process represented by 
\[
    \sum_{j=1}^\infty f_j\;dW_t^j
\]
where $\{W_t^j\}$ are a family of independent one-dimensional Brownian motions supported on a common canonical filtered probability space denoted as $(\Omega, \F, (\F_t), \P)$, and the $\{f_j\}$ are smooth, divergence free vector fields on $\T_\lambda^d$ satisfying the colouring condition
\[
    \frac{1}{2}\sum_{j=1}^\infty \|f_j\|_{L^2}^2 = \varepsilon < \infty.
\]

Weak solutions to Equation \eqref{eq: incompressible NSE} are called \textit{martingale solutions} and are the probabilistic analog of Leray-Hopf solutions. The existence of such solutions was first proven in \cite{bensoussan1973equations}. See Section \ref{sec: prelim martingale sols} for their definition and further information. As these solutions are random variables, they exhibit random fluctuations in the flow and a form of stochastic roughness. Furthermore, the Lagrangian trajectories are chaotic due to Stochastic Navier Stokes Equations possessing a positive Lyapunov exponent \cite{bedrossian2022lagrangian}. Solutions do not conserve kinetic energy along particular trajectories (i.e. in a deterministic sense), however, due to the martingale structure of the solutions, kinetic energy is conserved on average and for statistically stationary solutions there is a lack of dependence on the initial conditions. These kinds of features are all typically found in turbulent flows \cite{frisch1995turbulence}, which is why Equation \eqref{eq: incompressible NSE} has been studied as a model for fully developed turbulence in multiple situations. See \cites{inoue1979new, paret1997experimental, mattingly1998stochastic, bedrossian2019sufficient, boffetta2010evidence} and the references within for examples.\newline 

Of interest in this paper is the dynamics of energy in the invisicid limit. In 1922, Richardson \cite{richardson1922weather} conjectured that kinetic energy is not conserved in the invisicid limit and that there is a cascade of energy from the large scales where it is injected into the system to small scales where it may be dissipated by the viscosity. Kolmogorov provided a heuristic argument for the rate at which energy is transferred to the small scales for statistically stationary, isotropic flows in a series of papers in 1941 now known as K41 theory \cites{kolmogorov1941local, kolmogorov1941dissipation, kolmogorov1941degeneration}. Since then his theory has been extended to 2D flows by Batchelor \cite{batchelor1969computation}, Kraichnan \cite{kraichnan1967inertial}, and Fj{\o}rtoft \cite{fjortoft1953changes} and numerous experimental and simulation studies have agreed with their conjectures. Onsanger \cite{onsager1949statistical} expounded upon the ideas of Kolmogorov to provide an upper bound on the regularity of the velocity for which such an energy cascade can exist (See Section \ref{sec: 3d anomalous dissipation} for more details) \newline 

Within the mathematical literature, most of the attention surrounding the energy dynamics has been related to Onsanger's conjecture (see \cites{constantin1994onsager, eyink1994energy, duchon2000inertial, drivas2018onsager, drivas2022self} and the references within for a sample). Only relatively recently have mathematicians began to rigorously cement K41 theory in mathematical analysis. In the deterministic setting, in a landmark paper, Duchon and Robert\cite{duchon2000inertial} found a connection between Kolmogorov's 4/5 law (see Section \ref{sec: intro 3d laws} for details) and possible local singularities in the velocity. Similar results were found for passive scalar turbulent flux laws in \cites{warhaft2000passive, bedrossian2022batchelor, shraiman2000scalar}. In the stochastic setting, Bedrossian et. al. provided sufficient conditions for Kolmogorov's 4/3 and 4/5 laws on $\T^3$ for continuous (in space) solutions \cite{bedrossian2019sufficient} and in a later paper they provide similar sufficient conditions for the third order flux laws on $\T^2$ when there is a friction term to collect energy at large scales \cite{bedrossian2020sufficient}. However, their analysis only proves existence in the case that \textit{all} of the energy is dissipated at the smallest length scales and when the solutions are continuous in space. However, recent studies on thin flows shows the existence of split cascades \cites{musacchio2017split, GOLDBURG1997340}, where some energy moves towards the smallest scales to be dissipated and the rest moves towards scales near the size of the system. Similarly in 2D soap films, energy will move to the small scales as part of the enstrophy while the rest will move to large scales to create eddies \cites{rutgers1998forced, kraichnan1967inertial}. For these reasons, the author provides a rigorous mathematical characterization of Kolmogorov's third order flux laws to classify them even when energy remains at medium and large length scales. 

\subsection{3D Incompressible Flows}\label{sec: intro 3d big header}
Since the dynamics of the problem are slightly different when working in 2D vs 3D, we will examine them seperately. We begin with the 3D case as this has been historically the case due to it being slightly easier to analyze.

\subsubsection{Kolmogorov's Flux Laws}\label{sec: intro 3d laws}
In the 1930s and 40s, the physical literature on turbulence was revolutionized by Taylor, Richardson, Kolmogorov, and Obukhov to name a few. In order to experimentally measure the fluctuation in the velocity, Taylor considered the difference in velocity at two difference points: $\delta_hu^\nu(x) = u^\nu(x+h) - u^\nu(x)$ for all $h \in \R^3$, and averaged together multiple of these differences together \cite{frisch1995turbulence}. These combinations are refered to as structure functions. Using three axioms about the velocity field: 
\begin{enumerate}
    \item the flow is homogeneous and isotropic away from the boundary i.e. $\delta_hu^\nu$ is invariant under translations and rotations away from the boundary (in the case of a bounded domain)  
    \item the flow has mono-fractal scaling, i.e. the system is self similar with a single fractal scaling exponent $s \in \R$ such that $\delta_{\lambda \ell} u^\nu = \lambda^s\delta_\ell u^\nu$ for all $\lambda > 0$ and the increments $\lambda, \lambda \ell \ll \ell_{inj}$ where $\ell_{inj}$ is the scale of energy injection \cite{schaffner2015multifractal}
    \item there is an finite non-zero mean rate of dissipation $\varepsilon$ per unit mass independent of $\nu$\label{assumption: Kolmogorov AD}
\end{enumerate}
Kolmogorov provided a heuristic argument that there is an inertial range $\ell_\nu \ll \ell \ll \ell_{inj}$ such that
\[
    S_{vel, p}^\parallel(\ell) := \langle (\delta_{\ell n}u^\nu \cdot n)^p\rangle \sim C_p\varepsilon^{p/3}\ell^{p/3}
\]
Here $C_p$ is a fixed diffusion constant depending on the dimension and the power of $p$ (and independent of $\nu$), $\langle \cdot \rangle$ is some combination of ensemble, spatial, rotational, and temporal averaging and $S_{vel,p}^\parallel$ is the $p$th order universal longitudinal velocity structure function. Landau famously objected to Kolmogorov's second assumption of mono-fractal scaling due to intermittency effects, i.e. rare events and non-uniformity in the roughness of the velocity field, \cite{frisch1995turbulence}. Indeed, both of the first two assumptions have been found to be generally false due to such intermittency effects \cites{frisch1995turbulence, drivas2022self, bedrossian2019sufficient, jimenez2007intermittency} However such errors can be corrected for by including an intermittency corrector $\tau(p)$ to the exponent on the right hand side of the law. See \cite{jimenez2007intermittency} and the references within for more details regarding intermittency. Nevertheless, experimental results show that the $p = 3$ case correlates exceptionally close to Kolmogorov's prediction to the point that it is considered the only exact law in the physics literature \cites{drivas2022self}. This is because the $p=3$ case doesn't rely on the statistical assumptions of isotropy and mono-fractal scaling, but can be derived directly from the Navier Stokes equations via the Karman-Howarth-Monin (KHM) relations (see Section \ref{sec: prelim KHM eqs} for more details). \newline 

When $p = 3$ Kolmogorov's prediction is known as the $4/5$-law due to its diffusion coefficient: 
\begin{equation}\label{eq: 4/5 law}
    \langle (\delta_{\ell n}u^\nu \cdot n)^3 \rangle \sim -\frac{4}{5}\varepsilon\ell \quad \quad \ell_\nu \ll \ell \ll \ell_{inj}
\end{equation}
where $n$ it a unit vector, and the dissipation scale is given by $\ell_\nu = (\nu^3/\varepsilon)^{1/4}$ is the scale at which dissipation effects begin to dominate the flow. There is also a closely related $4/3$-law originally derived by Yaglom \cite{yaglom1949local} for passive scalars:
\begin{equation}\label{eq: 4/3 law}
    \langle |\delta_{\ell n}u^\nu|^2\delta_{\ell n}u^\nu \cdot n \rangle \sim \frac{-4}{3}\varepsilon\ell \quad \ell_{\nu} \ll \ell \ll \ell_{inj}
\end{equation}
In both instances, the negative sign on the flux indicates that energy is moving from large scales to small scales and one can see that the rate of the flux decreases proportional to the scale. 


\subsubsection{Anomalous Dissipation of Kinetic Energy}\label{sec: 3d anomalous dissipation}
Finally, we consider Kolmogorov's third ``axiom" which is also known as the anomalous dissipation of kinetic energy. Experimental evidence for such behaviour was first observed by Taylor \cite{taylor1935statistical} and can be thought of as a loss at large Reynold's numbers type assumption. The cause of this anomalous behaviour is thought to be due to the nonlinear term $(u^\nu\cdot \grad )u^\nu$, which is accelerating the movement of energy to small scales in such a way that $u^\nu$ losses its derivative in the invisicid limit. \newline 

To phrase this mathematically, first consider the determinstic case:
It is well known that Leray-Hopf solutions to the Navier-Stoke's equation satisfy an energy inequality of the form (see \cite{drivas2022self} for example)
\[
    \|u^\nu(T)\|_{L^2}^2 + \nu\int_0^T\|\grad u^\nu(t)\|_{L^2}^2 \;dt \leq \|u^\nu(0)\|_{L^2}^2 + \sum_{j=1}^\infty\int_0^T\langle f_j, u^\nu(t) \rangle \;dt
\]
Then typically for such deterministic solutions, anomalous dissipation is defined as there exists $\varepsilon > 0$ such that
\[
    \liminf_{\nu \to 0} \nu\int_0^T \|\grad u^\nu(t)\|_{L^2}^2\;dt \geq \varepsilon > 0
\]
Thus anomalous dissipation can be seen as a loss in regularity type assumption where the $L^2(0,T, \dot{H}^1)$ norm of $u^\nu$ must blow up like $1/\nu$. In a landmark paper by Duchon and Robert \cite{duchon2000inertial} it was shown that this inequality may be transformed into an equality if one includes a Radon measure $D(u^\nu)$ which is defined locally in terms of (possible) finite time singularities. Then one may consider a larger class of solutions $u^\nu$ where $\nu\int_0^T \|\grad u^\nu(t)\|_{L^2}^2\;dt \to 0$ but the invisicid limit consists of singularities on a set of positive measure. \newline

Similarly, for statistically stationary martingale solutions one can form the energy balance via Ito's lemma
\begin{align}\label{eq: stoch energy balance}
    \nu\E\fint_0^T\|\grad u^\nu(s)\|_\lambda^2\;ds + \E D(u^\nu) = \frac{1}{2}\sum_j\|f_j\|_\lambda^2 = \varepsilon
\end{align}
See the Appendix for the details for how to derive the energy balance \eqref{eq: stoch energy balance}. Additionally if the forcing has a deterministic component one can establish a similar bound, see \cite{chow2020zeroth} for details. 
We note that while the right hand side of the energy balance is independent of the viscosity, the deterministic regime's definition of anomalous dissipation does not apply here when $u^\nu$ is sufficiently regular so that $\E D(u^\nu) \equiv 0$. This is because there is nothing inherently nonlinear nor anomalous in the statement:
\[
    \nu\E\fint_0^T\|\grad u^\nu(s)\|_\lambda^2\;ds = \varepsilon
\]
as the same energy balance can be achieved by the stochastic Heat equation \cite{bedrossian2019sufficient}. Thus in the stochastic regime we define anomalous dissipation as
\begin{definition}\label{def: anomalous diffusion}
    We say that a sequence $\{u^\nu\}_{\nu \geq 0}$ of stationary martingale solutions satisfies weak anomalous dissipation if either
    \begin{equation}\label{eq: singularities occurs}
        \liminf_{\nu \to 0} \E D(u^\nu) > 0
    \end{equation}
    or there exists $N_\nu \to \infty$ such that
    \begin{equation}\label{eq: energy move to infinity}
        \liminf_{\nu \to 0} \nu\sum_{|k| \geq N_\nu}\E\fint_0^T|\widehat{\grad u^\nu}(k)|^2 > 0
    \end{equation}
\end{definition}
Notice that condition \eqref{eq: energy move to infinity} can be restated as in the invisicid limit, the family of solutions $u^\nu$ moves a sufficient amount of energy to small scales for $u^\nu$ to lose its $H^1$ regularity. 
\begin{remark}\label{remark: bedrossian implies WAD}
    Previously, the authors in \cite{bedrossian2019sufficient} and \cite{papathanasiou2021sufficient} defined a weak anomalous dissipation condition in terms of the Taylor micro-scale
    \[
        \lim_{\nu \to 0}\nu\E\|u^\nu(t)\|_{\lambda}^2 = 0 \quad \forall \;t
    \]
    Such a condition is sufficient to guarantee the existence of a sequence $N_\nu$ such that \eqref{eq: energy move to infinity} holds. For example, if $N_\nu^{-2} = o(\nu\E\|u^\nu\|_{L^2}^2)$  then
    \[
        \limsup_{\nu \to 0} \nu\sum_{|k| \leq N_\nu}\E\fint_0^T|\widehat{\grad u^\nu}(k)|^2 \leq \limsup_{\nu \to 0}N_\nu^2 \nu\sum_{|k| \leq N_\nu}\E|\widehat{u^\nu}|^2 \leq \limsup_{\nu \to 0}N_\nu^2 \nu\E\|u^\nu\|_{L^2}^2 = 0
    \]
    Therefore
    \[
        \liminf_{\nu \to 0} \nu\sum_{|k| \geq N_\nu}\E\fint_0^T|\widehat{\grad u^\nu}(k)|^2 = \varepsilon - \limsup_{\nu \to 0} \nu\sum_{|k| \leq N_\nu}\E\fint_0^T|\widehat{\grad u^\nu}(k)|^2 = \varepsilon > 0
    \]
\end{remark}

It is an interesting question whether the condition in \eqref{eq: singularities occurs} is true for some flows. Sufficient conditions for $\ds\liminf_{\nu \to 0}D(u^\nu) \equiv 0$ have been found in the deterministic case, see \cite{leslie2018conditions} for work in this direction. To the authors no such analysis has been carried out in the stochastic regime, but similar results should hold. Additionally, one could ask under what regularity do solutions satisfy condition \eqref{eq: energy move to infinity} and converge to a solution to the stochastic Euler equation. While this is merely the stochastic version of Onsanger's conjecture, we remark that to the best of our knowledge no such analysis has been done (in the stochastic case) but most likely the answer relies on a translation of the work by Eyink \cite{eyink1994energy} and Constantin et.al. \cite{constantin1994onsager} to the stochastic framework.

\subsection{2D Incompressible Flows}\label{sec: intro 2d laws}
Now we consider the problem's dynamics when restricted to two dimensions. 
In two dimensions one could extend the three dimensional results directly for particularly rough solutions. However, if the solution is nice enough that the initial velocity has finite expected enstrophy: 
\[
    \E\|\omega^\nu(0)\|_\lambda^2 < \infty \quad \omega^\nu(t) = \partial_1u_2^\nu(t) - \partial_2 u_1^\nu(t)
\]
then the 2D dynamics are largely disparate from the 3D dynamics. The main culprit for this difference is the lack of vortex stretching in 2D resulting in a well-behaved vorticity equation.\newline 

Consider the vorticity equation in two dimensions:
\begin{align}\label{eq: vorticity eq}
    d\omega^\nu + (u^\nu \cdot \grad)\omega^\nu\;dt &= \nu\Delta \omega^\nu\;dt + \sum_j \curl f_j\;dW_t^j\\
    u^\nu &= \grad^\perp (-\Delta)^{-1}\omega^\nu\nonumber
\end{align}
Here $\grad^\perp := (-\partial_2, \partial_1)$ and we call the second relation $u^\nu = \grad^\perp (-\Delta)^{-1}\omega^\nu$ the Biot-Savart law. As there is no vortex stretching, the vorticity equation is an advection-diffusion equation which is fairly well behaved. First, it is known that one can construct mild solutions by pushing the initial condition forward in time via a semi-group operator (see Section \ref{sec: prelim mild sols} for the definition and futher details). 
Second, the total enstrophy is conserved for all time. This results in a second invariant for the system which \textit{should} be conserved in the invisicid limit in some sense. In order to guarantee the existance of stationary measures we make th following compactness assumption:
\begin{assumption}\label{assumption: forcing assumptions}
There exists a constant $C > 0$ such that
\begin{align*}
    \sup_{\lambda \in (1,\infty)}\sum_j\|\grad^3 f_j\|_\lambda^2 &\leq C
\end{align*}
\end{assumption}
This is a purely mathematical assumption used to ensure the existence of stationary measures for the 2D vorticity equation. \newline

When extending Kolmogorov's K41 theory to 2D, Kraichnan \cite{kraichnan1967inertial}, Batchelor \cite{batchelor1969computation}, Leith \cite{}, and Fjortoft \cite{fjortoft1953changes} independently observed two different inertial zones, one for each conserved quantity. 
It was observed that now energy moves toward large scales in an inverse cascade, while enstrophy moves to small scales via a direct cascade. Such situations are frequently called dual cascades and are generally known to exist in systems with two conserved quantities\cite{bedrossian2020sufficient}. Specifically the direct cascade flux laws for the velocity (and the vorticity as predicted by Eyink \cite{eyink1996exact}) can be informally stated as there exists a dissipation scale $\ell_\nu$ such that over the inertial range $\ell_\nu \ll \ell \ll \ell_{inj}$
\begin{align*}
    \langle |\delta_{\ell n}\omega^\nu|^2\delta_{\ell n}u^\nu \cdot n\;\rangle \sim -2\eta\ell\\
    \langle |\delta_{\ell n}u^\nu|^2\delta_{\ell n}u^\nu \cdot n\rangle \sim \frac{1}{4}\eta\ell^3\\
    \langle (\delta_{\ell n}u^\nu \cdot n)^3 \rangle \sim \frac{1}{8}\eta\ell^3
\end{align*}
Moreover, there is a \textit{second} inertial range $\ell_{inj} \ll \ell \ll \tilde{\ell}_{nu} \leq \lambda$ for the inverse cascade such that
\begin{align*}
    \langle |\delta_{\ell n}u^\nu|^2\delta_{\ell n}u^\nu \cdot n\rangle \sim 2\varepsilon\ell\\
    \langle (\delta_{\ell n}u^\nu \cdot n)^3 \rangle \sim \frac{3}{2}\varepsilon\ell
\end{align*}

A simple heuristic argument for the existence of two separate inertial ranges with different flux coefficients can be seen from the two associated balance laws for the system:
\begin{align*}
    \nu\E\int_0^T\|\grad \omega^\nu\|_\lambda^2 \;dt= \frac{1}{2}\sum_{j=1}^\infty \|\curl f_j\|_\lambda^2 =: \eta < \infty\\
    \nu\E\int_0^T\|\omega^\nu\|_\lambda^2 \;dt= \frac{1}{2}\sum_{j=1}^\infty \|f_j\|_\lambda^2 = \varepsilon < \infty
\end{align*}
The first comes from Ito's lemma and computing the energy balance for stationary solutions to the vorticity equation, while the second one is the energy balance for stationary solutions to the Navier-Stokes equations after we applied the Biot-Savart Law. As neither of the right hand sides depend on the viscosity, it can be seen that at least one wave-number of $\omega^\nu$ must become singular. Moreover, if there are two singular wave-numbers then flux of energy to these scales must occur at different rates as one is moving like $|k|^2\E|\widehat{\omega^\nu}(k)|^2$ and the other is moving like $\E|\widehat{\omega^\nu}(k)|^2$. \newline

While the existence of both cascades simultaneously had been conjectured for a while, only recently in \cite{boffetta2010evidence} did researchers observe such a phenomena in soap films. This is because the inverse cascade is somewhat unstable as it is due to the difference in the amount of energy at large scales and how much energy the force pushes towards this end of the spectrum. Moreover, as we show in Corollary \ref{corollary: isolated cascades} if the enstrophy is does not blow up fast enough at small wave-numbers then there is no inverse cascade. For these reasons most works regarding inverse cascades include a friction force to dampen the large wave-number terms and collect the energy at small wave-numbers or a hyper/hypo-viscosity term to heighten the effects of the viscous dissipation. See \cites{rutgers1998forced, paret1997experimental} and the references within for more information. Within the mathematics literature far fewer works have been produced compared to the three dimensional case. However, we note that a major inspiration for this paper was the work by Bedrossian et.al. \cite{bedrossian2020sufficient} where they showed sufficient conditions for the existence of a dual cascade to the Navier-Stoke's equations with a frictional term of the form: $\alpha (-\Delta)^{-2\gamma}u^\nu$. In this paper, we generalize their argument to the case without friction.

\subsection{Main Results}\label{sec: intro main results}
Inspired by the ideas in \cite{bedrossian2019sufficient}, \cite{bedrossian2020sufficient}, and \cite{duchon2000inertial} we characterize the direct cascades in terms of the (possible) finite time singularities and the amount of the energy which moves out to infinite wave-numbers. In contrast we characterize an inverse cascade as the amount of energy moving towards the wave-numbers near 0 due to the forcing. On top of reestablishing the results in each of the highlighted works, we can also provide necessary and sufficient conditions for split cascades and isolated cascades without frictional forces in both 2D and 3D. \newline 

First we characterize a direct cascade in both $\T^2$ and $\T^3$ separately due to their differences covered before. 
\begin{theorem}[3D Direct Cascade Characterization]\label{thm: 3d Direct cascade characterization}
    Suppose $\{u\}_{\nu > 0}$ is a sequence of statistically stationary solutions to the statistically forced Navier-Stokes equations \eqref{eq: incompressible NSE} on $\T_\lambda^3$. 
    There exists $N_\nu \geq 1$ such that $\ds\lim_{\nu \to 0}N_\nu = \infty$ and 
    \[
        \liminf_{\nu \to 0}\Big(\nu \sum_{|k| \geq N_\nu}\E\fint_0^T|\widehat{\grad u}|^2 + \E D(u^\nu)\Big) = \varepsilon^*
    \]
    if and only if there exists $\ell_\nu \in (0, 1)$ such that $\ds\lim_{\nu \to 0}\ell_\nu = 0$ and 
    \begin{align}
        \lim_{\ell_I \to 0}\limsup_{\nu \to 0}\sup_{\ell \in [\ell_\nu, \ell_I]}\Big|\frac{1}{\ell}\E\fint_{S^{2}}\fint_0^T\fint_{\T_\lambda^3}|\delta_{\ell n}u^\nu|^2 (\delta_{\ell n}u^\nu\cdot n) \;dxdtdS(n) + \frac{4}{3}\varepsilon^*\Big| = 0\label{eq: 3d S0 limit}\\
        \lim_{\ell_I \to 0}\limsup_{\nu \to 0}\sup_{\ell \in [\ell_\nu, \ell_I]}\Big|\frac{1}{\ell}\E\fint_{S^{2}}\fint_0^T\fint_{\T_\lambda^3} (\delta_{\ell n}u^\nu\cdot n)^3 \;dxdtdS(n) + \frac{4}{5}\varepsilon^*\Big| = 0\label{eq: 3d S long limit}
    \end{align}
\end{theorem}
Similarly in 2D:
\begin{theorem}[2D Direct Cascade Characterization]\label{thm: 2d direct cascade characterization}
    Let $\{u\}_{\nu > 0}$ be a sequence of statistically stationary solutions to the stochastically forced Navier-Stokes equations\eqref{eq: incompressible NSE} and the 2D vorticity equation \eqref{eq: vorticity eq} on $\T_\lambda^2$. 
    There exists $N_\nu \geq 1$ such that $\ds\lim_{\nu \to 0}N_\nu = \infty$ and 
    \[
        \liminf_{\nu \to 0}\nu \sum_{|k| \geq N_\nu}\E|\widehat{\grad \omega^\nu}|^2 = \eta^*
    \]
    if and only if there exists $\ell_\nu \in (0, 1)$ such that $\ds\lim_{\nu \to 0}\ell_\nu = 0$ and 
    \begin{align}
        \lim_{\ell_I \to 0}\limsup_{\nu \to 0}\sup_{\ell \in [\ell_\nu, \ell_I]}\Big|\frac{1}{\ell^3}\E\fint_{S^{1}}\fint_0^T\fint_{\T_\lambda^2}|\delta_{\ell n}\omega^\nu|^2 (\delta_{\ell n}u^\nu\cdot n) \;dxdtdS(n) + 2\eta^*\Big| = 0\label{eq: vorticity limit}\\
        \lim_{\ell_I \to 0}\limsup_{\nu \to 0}\sup_{\ell \in [\ell_\nu, \ell_I]}\Big|\frac{1}{\ell^3}\E\fint_{S^{1}}\fint_0^T\fint_{\T_\lambda^2}|\delta_{\ell n}u^\nu|^2 (\delta_{\ell n}u^\nu\cdot n) \;dxdtdS(n) - \frac{1}{4}\eta^*\Big| = 0\label{eq: 2d S0 limit}\\
        \lim_{\ell_I \to 0}\limsup_{\nu \to 0}\sup_{\ell \in [\ell_\nu, \ell_I]}\Big|\frac{1}{\ell^3}\E\fint_{S^{1}}\fint_0^T\fint_{\T_\lambda^2} (\delta_{\ell n}u^\nu\cdot n)^3 \;dxdtdS(n) - \frac{1}{8}\eta^*\Big| = 0\label{eq: 2d S long limit}
    \end{align}
\end{theorem}
Here the limit $\ell_I \to 0$ should \textit{not} be confused with injection scale of energy moving to small scales. Instead $\ell_I$ is acting as a dummy variable signifying when the scales go to 0. This is similar to the limits shown in \cites{frisch1995turbulence, bedrossian2019sufficient, bedrossian2020sufficient, papathanasiou2021sufficient}. 
\begin{remark}
    In both 2D and 3D, the size of the torus $\lambda$ is insignificant. This is because the direct cascade limit localizes the structure function. For this reason, one can also extend these ideas to domains with boundaries by localizing each of the quantities. See \cite{papathanasiou2021sufficient} for details.
\end{remark}

While the direct cascade, is significantly different between $\T^2$ and $\T^3$, the existence of the inverse cascade is nearly identical in both dimensions. Moreover, the inverse cascade is very sensitive to the size of the domain and requires $\lambda \to \infty$. This is because while the direct cascade localizes the structure functions, the inverse cascade globalizes them.
\begin{theorem}[Inverse Cascade Characterization]\label{thm: inverse cascade characterization}
    Suppose that $\lambda = \lambda(\nu) < \infty$ is a continuous monotone increasing function such that $\lim_{\nu \to 0}\lambda = \infty$. Let $\{u\}_{\nu > 0}$ be a sequence of statistically stationary solutions to \eqref{eq: incompressible NSE} with mean-zero, divergence free forcing correlations $f_j$.  
    There exists a decreasing sequence $M_\nu$ satisfying $\ds\lim_{\nu \to 0}M_\nu = 0$ such that 
    \[
        \liminf_{\nu \to 0}\nu\sum_{|k| \leq M_\nu}\E\fint_0^T|\widehat{\grad u^\nu}|^2 = \varepsilon^*
    \]
    if and only if there exists an increasing sequence $\tilde{\ell}_\nu \in (1, \lambda)$ satisfying $\ds\lim_{\nu \to 0}\tilde{\ell}_\nu = \infty$ such that:
    \begin{align}
        \lim_{\ell_I \to \infty}\limsup_{\nu \to 0}\sup_{\ell \in [\ell_I, \tilde{\ell}_\nu]}\Big|\frac{1}{\ell}\E\fint_{S^{d-1}}\fint_0^T\fint_{\T_\lambda^d}|\delta_{\ell n}u^\nu|^2 (\delta_{\ell n}u^\nu\cdot n) \;dxdtdS(n) - \gamma_d\varepsilon^*\Big| = 0\label{eq: inverse 4/3 law}\\
        \lim_{\ell_I \to \infty}\limsup_{\nu \to 0}\sup_{\ell \in [\ell_I, \tilde{\ell}_\nu]}\Big|\frac{1}{\ell}\E\fint_{S^{d-1}}\fint_0^T\fint_{\T_\lambda^d} (\delta_{\ell n}u^\nu\cdot n)^3 \;dxdtdS(n) - \kappa_d\varepsilon^*\Big| = 0 \label{eq: inverse 4/5 law}
    \end{align}
    where the coefficients are given by
    \[
        \gamma_d = \begin{cases}
            2 & d = 2\\
            4/3 & d = 3
        \end{cases}\quad \quad \kappa_d = \begin{cases}
            3/2 & d = 2\\
            4/5 & d = 3
        \end{cases}
    \]  
\end{theorem}

\begin{remark}
    Notice that the difference in coefficients $\gamma_d$ and $\kappa_d$ when $d = 2,3$ are due to the difference in dimension but the proof is otherwise  exactly the same.
\end{remark}
\begin{remark}
    Due to the Biot-Savart law, in 2D, the equivalent condition for the inverse cascade can be reformulated as 
    \[
        \liminf_{\nu \to 0}\nu\sum_{|k| \leq M_\nu}\E|\widehat{\omega^\nu}|^2 = \varepsilon^* + \frac{1}{2}\sum_j|\widehat{f_j}(0)|^2
    \]
\end{remark}
\begin{remark}
    Notice that the flux laws for the direct cascades are defined in the energy/enstrophy dissipated by the viscosity at the infinite wave-numbers and any potential finite time local singularities. Whereas, as an artifact of the proof to Theorem \ref{thm: inverse cascade characterization} an inverse cascade is a difference between the average force and the energy dissipation at the largest scales. However the exact balance between these quantities differs between the flux laws \eqref{eq: inverse 4/3 law} and \eqref{eq: inverse 4/5 law}. Moreover, for \eqref{eq: inverse 4/5 law}, the balance even changes with dimension.  
\end{remark}

\begin{remark}
    This characterization is in agreement with the weak anomalous dissipation results for direct cascades from \cites{bedrossian2019sufficient} in the case where $\E D(u^\nu) \equiv 0$. By this we mean, that if the energy is blows up at most like $1/\nu$, then a sufficient estimate for $N_\nu$ is given by:
    \[
        N_\nu^2 = o((\nu\E\|u^\nu\|_{\lambda}^2)^{-1})
    \]
    See Remark \ref{remark: bedrossian implies WAD} for details.
    Furthermore, when $\E D(u^\nu) \equiv 0$ then existence of such a sequence $N_\nu$ signifies that $\E\||\grad|u_\nu\|_{L^2} = \infty$
\end{remark}

\begin{remark}
    Notice that if one knows the spectrum of $u^\nu$ through a scaling argument or an ad-hoc assumption, then one can quickly find the sequences $N_\nu$ and $M_\nu$ such that either an inverse or direct cascade can occur.
\end{remark}

Note that essentially one can view the existence of a inverse cascade as there is a singularity at $k=0$ in the spectrum of $\lim_{\nu \to 0}u^\nu$ while the existence of the direct cascade is the existence of a singularity at $k = \infty$. However note that the strength of such singularities (in terms of the energy spectrum) is particularly vague, i.e. what is the order of the singularity, and there is no requirement on the uniqueness of such spectral singularities. Thus solutions may theoretically have both an inverse and direct cascades simultaneously. Physical experiments in which such behaviour is known to occur appears in thin film turbulence \cites{musacchio2017split, GOLDBURG1997340} and soap films \cite{rutgers1998forced}. The existence of such split cascades in the velocity also allows for the dual cascades (in the vorticity) seen in 2D turbulence \cite{kraichnan1967inertial}. Thus we provide some example conditions which imply the existence of split and isolated cascades. 

\begin{corollary}
     Suppose $\{u^\nu\}_{\nu > 0}$ is a sequence of statistically stationary solutions to \eqref{eq: incompressible NSE} on $\T^3$, then 
     \[
        \lim_{\nu \to 0}\nu\E\|u^\nu(t)\|_{L^2}^2 = 0 \quad \forall \; t
     \]
     is a sufficient condition for an isolated direct cascade in the sense of the flux laws \eqref{eq: 4/3 law} and \eqref{eq: 4/5 law}
\end{corollary}
This follows directly from \ref{thm: 3d Direct cascade characterization} and Remark \ref{remark: bedrossian implies WAD}.

\begin{corollary}
    Suppose $\{u^\nu\}_{\nu > 0}$ is a sequence of smooth statistically stationary solutions to \eqref{eq: incompressible NSE} on $\T^3$, and there exists $j$ such that $|\widehat{f_j}(0)| > 0$. Assume there exists a wave-number $c \in (0, \infty)$ such that 
    \[
        \lim_{\nu \to 0}\nu\sum_{|k| \geq c}\E|\widehat{u^\nu}(k)|^2 = 0 \quad  \text{ and } \lim_{\nu \to 0}\nu\sum_{|k| \leq c}\E|\widehat{\grad u^\nu}(k)|^2 < \varepsilon
    \]
    If there is no flux of energy to large scales in the sense of \eqref{eq: inverse 4/3 law} and \eqref{eq: inverse 4/5 law} with $\varepsilon^* = 0$, then there is a split cascade (i.e. there exists both sequences $N_\nu$ and $M_\nu$).
\end{corollary}
\begin{proof}
    By assumption the zero energy flux of the structure functions at large scales, it follows from Theorem \ref{thm: inverse cascade characterization} that there exists $M_\nu \to 0$ such that 
    \[
        \liminf_{\nu \to 0}\nu \sum_{|k| \leq M_\nu}\E\fint_0^T |\widehat{\grad u^\nu}|^2 = \frac{1}{2}\sum_{j}|\widehat{f_j}(0)|^2 > 0
    \]
    Now choose $N_\nu \to \infty$ such that $\lim_{\nu \to \infty}\nu|N_\nu|^2\sum_{|k| \geq c}\E|\widehat{u^\nu}(k)|^2 = 0$. Then 
    \[
        \liminf_{\nu \to 0}\nu\sum_{|k| \geq N_\nu}\E\fint_0^T|\widehat{\grad u^\nu}|^2 = \varepsilon - \limsup_{\nu \to 0}\nu\sum_{|k| \leq c}\nu\E\fint_0^T|\widehat{\grad u^\nu}|^2 = \varepsilon^* > 0 
    \]
\end{proof}
Finally in the case of 2D flows
\begin{corollary}
    Suppose $\{u^\nu\}_{\nu > 0}$ is a sequence of statistically stationary solutions to both the Navier-Stokes equations \eqref{eq: incompressible NSE} and the vorticity equation \eqref{eq: vorticity eq} on $\T^2$ with local forcing (i.e. $\widehat{f_j}(0) = 0$ for all $j$). Assume there exists a wave-number $c \in (0, \infty)$ such that 
    \[
        \lim_{\nu \to 0}\nu\sum_{|k| \geq c}\E|\widehat{\omega^\nu}(k)|^2 = 0 \quad  \text{ and } \lim_{\nu \to 0}\nu\sum_{|k| \leq c}\E|\widehat{\grad \omega^\nu}(k)|^2 = 0
    \]
    Then there is a dual cascade in the sense of Theorems \ref{thm: 2d direct cascade characterization} and \ref{thm: inverse cascade characterization}.
\end{corollary}
\begin{proof}
    Pick $N_\nu \to \infty$ such that for sufficiently large $N_\nu$ $N_\nu^2 \nu\sum_{c \leq |k| \leq N_\nu} \E\fint_0^T |\widehat{\omega^\nu}|^2 \to 0$. Then 
    \begin{align*}
        \nu\sum_{|k| \leq N_\nu}\E\fint_0^T |\widehat{\grad \omega^\nu}|^2 &\leq \nu\sum_{|k| \leq c}\E\fint_0^T |\widehat{\grad \omega^\nu}|^2 + \nu\sum_{c \leq |k| \leq N_\nu}\E\fint_0^T N_\nu^2|\widehat{\omega^\nu}|^2 \to 0 \quad \text{ as } \nu \to 0
    \end{align*}
    Hence by the enstrophy balance \eqref{eq: enstrophy balance} 
    \[
        \ds\liminf_{\nu \to 0}\nu\sum_{|k| \geq N_\nu}\E|\widehat{\grad \omega^\nu}|^2 = \eta > 0.
    \]
    Similarly, for the inverse cascade we pick $M_\nu \to 0$ such that for $M_\nu$ sufficiently small it holds that $\nu \sum_{M_\nu \leq |k| \leq c}\frac{|k|^2}{M_\nu^2}\E\fint_0^T|\widehat{\omega^\nu} \to 0$ as $\nu \to 0$. Then
    \[
        \nu\sum_{|k| \geq M_\nu}\E\fint_0^T |\widehat{\omega^\nu}|^2 \geq \nu\sum_{|k| \geq c}\E\fint_0^T |\widehat{\omega^\nu}|^2 + \nu\sum_{M_\nu \leq |k| \leq c}\frac{|k|^2}{M_\nu^2}\E\fint_0^T |\widehat{\omega^\nu}|^2 \to 0 \quad \text{ as } \nu \to 0
    \]
    Thus by the energy balance \eqref{eq: energy balance} (after applying the Biot-Savart law) it follows that 
    \[
        \ds\liminf_{\nu \to 0}\nu\sum_{|k| \leq M_\nu}\E|\widehat{\omega^\nu}|^2 = \varepsilon > 0
    \]
\end{proof}

\subsection{Notation Conventions}\label{sec: intro notation}
We denote the averaged $L^2$-norm in space as $\|f\|_{\lambda} := \Big(\fint_{\T_\lambda^d} |f(x)|^2\;dx\Big)^{1/2}$ and will the notation
\[
    L^p := L^p(\T_\lambda^d), \quad H^s := W^{s,2}(\T_\lambda^d)
\]
to denote the $L^p$ and $H^s$ spaces over $\T_\lambda^d$ when the dimension $d$ is clear. 
Furthermore, we will use the following Fourier analysis conventions:
\[
    \widehat{f}(k) = \fint_{\T_\lambda^d}f(x)e^{-ix\cdot k}\;dx \quad f(x) = \frac{1}{\lambda^d}\sum_{k \in \frac{2\pi}{\lambda}\Z^d} \widehat{f}(k)e^{ix\cdot k}
\]

Finally, we mention the use of component free tensor notation. For example, given two vectors $a$ and $b$ we will denote the tensor product $(a \otimes b)_{ij} = a_ib_j$. 
Also for the sake of making certain expressions simpler we will use the Einstein summation convention: $a_jb_j = \sum_{j=1}^da_jb_j$. 
Lastly, for two given rank-two tensors $A$ and $B$ we define the Frobenius product and norm as $A:B = A_{ij}B_{ij} = \sum_{i}\sum_j A_{ij}B_{ij}$ and $|A| = \sqrt{A:A}$.

\section{Preliminaries}\label{sec: prelims}
\subsection{Stationary Martingale Solutions}\label{sec: prelim martingale sols}
 As previously mentioned, one can construct weak solutions to the Navier Stokes equations \eqref{eq: incompressible NSE} called martingale solutions. They are the probabilistic analog of Leray-Hopf solutions which also happens to be a square-integrable martingale. By this we mean that a martingale solutions consist of a $d$-dimensional stochastic basis $(\Omega, (\F_t)_{t \in [0,T]}, \P, \{W_k\})$ and a process $u^\nu: \Omega \times [0,T] \to L_{\Div}^2$ which satisfies the Navier Stokes Equations in the sense of distributions and for all $s,t \in [0,T]$
 \[
    \E[u^\nu(t) \mid \F_s)] = u^\nu(s) \quad \quad s \leq t
 \]
 and $\E\|u^\nu(t)\|_{L^2} < \infty$. It is worth noting that what makes this probabilistically weak is that the stochastic basis must be solved for simultaneously to the process $u^\nu$ which means that the forcing
 \[
    \sum_{j}f_j\;dW_t^j
\]
is an \textit{output} rather than an input to the system. Hence in order to prescribe a particular forcing on the system we must do so through the choice of the spatial correlation vector fields $f_j$, which is slightly different from the deterministic setting. 
Other remarks to be aware of is that two martingale solutions to \eqref{eq: incompressible NSE} may exist on different stochastic basis and therefore have different trajectories for a particular $\omega \in \Omega$. Hence uniqueness is usually considered in the sense of the law of the process $u^\nu$. \newline 

The existence of martingale solutions to the stochastically forced Navier-Stokes Equations
\eqref{eq: incompressible NSE} have been known since the 1970s with the work of Bensoussan\cite{bensoussan1973equations}. Such solutions are rigorously defined as:
\begin{definition}\label{def: martingale sols}
    A $d$ dimensional stochastic basis $(\Omega, (\F_t)_{t \in [0,T]}, \P, \{W_k(t)\}_{t \in [0,T]})$ along with a complete right-continuous filtration and an $\F_t$-progressively measureable stochastic process $u^\nu:\Omega \times [0,T] \to L_{\Div}^2(\T_\lambda^d)$ is called a martingale solution to \eqref{eq: incompressible NSE} on $[0,T]$ if 
    \begin{itemize}
        \item the trajectories of $u^\nu$ (i.e. the sample paths $u^\nu(\cdot, t)$ for $t \in [0,T]$) belong to 
        \[
            C_t\dot{H}_x^\alpha(\T_\lambda^d)\cap L_t^\infty L_\Div^2(\T_\lambda^d)\cap L_t^2H_\Div^1(\T_\lambda^d),\quad \text{ for some $\alpha < 0$;}
        \]
        \item for all $t \in [0,T]$ and all $\phi \in C_\Div^\infty$ the following identity holds $\P$ almost surely:
        \begin{align}\label{eq: weak energy identity}
            \int_{\T^d}u^\nu(t) \cdot \phi\;dx &+ \nu \int_0^t \int_{\T^d}\grad u^\nu(x):\grad \phi \;dxds - \int_0^t \int_{\T^d}u^\nu(s)\cdot (u^\nu(s)\cdot \grad)\phi \;dxds\\
            &= \int_{\T^d}u^\nu(0)\cdot \phi\;dx + \sum_j\int_0^t\int_{\T^d}f_j\cdot \phi\;dxdW^j(s)\nonumber
        \end{align}
    \end{itemize}
\end{definition}
In this paper, we will call a process $u^\nu$ a martingale solution without reference to its stochastic basis unless there is a risk of confusion. 

Similar to the deterministic literature, several years after the existence of Leray-Hopf solutions did the existence of global in time weak solutions began to emerge. In the context of martingale solutions these global in time solutions are called statistically stationary martingale solutions as the law is invariant under time translations. They were first constructed in 1995 in the landmark paper by Flandoli and  Gatarek\cite{flandoli1995martingale} where they are defined as
\begin{definition}\label{def: stationary martingale sol}
    A $d$-dimensional stochastic basis $(\Omega, (\F_t)_{t \in [0,T]}, \P, \{W_k(t)\}_{t \in [0,\infty)})$ along with a complete right-continuous filtration and an $\F_t$-progressively measureable stochastic process $u^\nu:\Omega \times [0,\infty) \to L_{\Div}^2(\T_\lambda^d)$ is called a martingale solution to \eqref{eq: incompressible NSE} on $[0,\infty)$ if 
    \begin{itemize}
        \item for each $T \geq 0$, the trajectories of $u^\nu|_{[0,T]}$ belong to 
        \[
            C_t\dot{H}_x^\alpha(\T_\lambda^d)\cap L_t^\infty L_\Div^2(\T_\lambda^d)\cap L_t^2H_\Div^1(\T_\lambda^d),\quad \text{ for some $\alpha < 0$;}
        \]
        \item the path of $u(\cdot + s) = u(\cdot)$ in law for all $s \geq 0$;
        \item for all $t\geq 0$ and all $\phi \in C_\Div^\infty$ identity \eqref{eq: weak energy identity} holds $\P$ almost surely.
    \end{itemize}
\end{definition}
Again, we will simply call a process $u^\nu$ a stationary martingale solution and omit reference to the associated stochastic basis. 

Next, we consider the energy in the system. Specifically, we consider the total kinetic energy:
\[
    \frac{1}{2}\|u^\nu\|_\lambda^2 = \frac{1}{2}\fint_{\T_\lambda^d}|u^\nu|^2\;dx
\]
and note that for stationary martingale solutions, the average kinetic energy is conserved for all time:
\[
    \frac{1}{2}\E\|u^\nu(t)\|_\lambda^2 = \frac{1}{2}\E\|u^\nu(s)\|_{\lambda}^2 \quad \forall\; t,s \in [0, \infty)
\]
Using standard results regarding Leray-Hopf solutions, one can quickly extend the ideas in \cite{flandoli1995martingale} to prove an energy inequality where the average energy dissipation due to viscosity is bounded by the total energy injected into the system by the force. However, this inequality can be made into an equality if one accounts for the dissipation of energy due to (possible) local singularities in the flow (such as shocks) \cite{duchon2000inertial}.
\begin{theorem}\label{thm: energy balance}
    Suppose $u^\nu$ is a statistically stationary martingale solution to \eqref{eq: incompressible NSE} in the sense of Definition \ref{def: stationary martingale sol}. Then the following energy equality holds:
    \begin{equation}\label{eq: energy balance}
        \nu\E\fint_0^T\|\grad u^\nu(s)\|_\lambda^2\;ds + \E D(u^\nu) = \frac{1}{2}\sum_j\|f_j\|_\lambda^2 = \varepsilon
    \end{equation}
    where $D(u^\nu)$ is a Radon measure defined in terms of the local smoothness of the trajectory of $u^\nu$. Moreover, $D(u^\nu) \equiv 0$ when $d = 2$ (i.e. 2D flow). 
\end{theorem}
This result is relatively well-known in the deterministic setting \cite{drivas2022self} but the author could find no reference to it within the stochastic literature. Therefore, the proof of this result can be found in the Appendix.

\subsubsection{Mild solutions in 2D}\label{sec: prelim mild sols}
Due to the lack of vortex stretching in $\R^2$, the vorticity equation \eqref{eq: vorticity eq} is an advection-diffusion equation and is fairly well-behaved. The improved behaviour of solutions, means that the martingale solutions can be upgraded from Leray-Hopf analogs to mild solutions which are acting under a semi-group operator \cite{bedrossian2020sufficient}.
Such mild solutions are defined as 
\begin{definition}\label{def: mild vorticity solution}
    Given a complete filtered probability space $(\Omega, \F, (\F_t)_{t \in [0,T]}, \P)$, a mild solution $(\omega_t^\nu)$ to \eqref{eq: vorticity eq} is an $\F_t$ adapted process $\omega^\nu :[0,T] \times \Omega \to L^2(\T_\lambda^2)$ satisfying 
    \begin{align*}
        \omega^\nu(t) = e^{-\nu\Delta t}\omega_0 - \int_0^te^{-\nu\Delta (t-s)}(u^\nu(s) &\cdot \grad \omega^\nu(s))\;ds + \sum_{j}\int_0^t e^{-\nu\Delta (t-s)}\curl f_j\;dW_s^j\\
        u^\nu(t) &= \grad^\perp (-\Delta)^{-1}\omega^\nu(t)
    \end{align*}
\end{definition}
One well known (see \cite{kuksin2012mathematics} for instance) result about such mild solutions is:
\begin{proposition}\label{proposition: exist stationary mild sols}
    Suppose that $\varepsilon$ and $\eta$ are both independent of $\lambda \geq 1$ and that Assumption \ref{assumption: forcing assumptions} holds. Then for all $\nu > 0$ and $\lambda \geq 1$, the vorticity equation \eqref{eq: vorticity eq} admits a global in time, $\P$ almost surely unique, mild solution $(\omega^\nu)$ with initial data $\omega_0^\nu$. Moreover, $(\omega^\nu)$ defines a Feller Markov process which has at least one stationary probability measure $\mu$ supported on $W^{3,2}(\T_\lambda^2)$. That is, a measure satisfying the following: for all bounded measurable $\phi:L^2 \to \R$ and $t \geq 0:$
    \[
        \int_{L^2(\T_\lambda^2)}\E\phi(\omega^\nu(t))\mu(d\omega_0^\nu) = \int_{L^2(\T_\lambda^2)}\phi\;d\mu.
    \]
\end{proposition}
\begin{definition}\label{def: stationary vorticity sol}
    A process $(\omega^\nu)$ which is a mild solution to the vorticity equation \eqref{eq: vorticity eq} is said to be a statistically stationary mild solution if for all $s > 0$, $\omega^\nu(t) = \omega^\nu(t+s)$ in law on $C(\R_+; L^2(\T_\lambda^2))$.
\end{definition}
Constructing such a statistically stationary mild solution is a simple corrollary of Proposition \ref{proposition: exist stationary mild sols}. We sketch the argument here:
Let $\mu$ be the stationary measure found in Proposition \ref{proposition: exist stationary mild sols}. Given an initial condition $\omega_0^\nu$ distributed according to $\mu$, let $\omega^\nu(t)$ be the push-forward of $\omega_0^\nu$ under the semi-group operator $e^{-\nu\Delta t}$ according to Definition \ref{def: mild vorticity solution}. As $\mu$ is stationary, so is the law of $\omega^\nu(t)$, hence $\omega^\nu(t)$ is a statistically stationary mild solution to \eqref{eq: vorticity eq}. A simple consequence for such stationary solutions and Ito's lemma is that 
\begin{equation}\label{eq: enstrophy balance}
    \nu\E\fint_0^T\|\grad \omega^\nu(s)\|_\lambda^2\;ds = \frac{1}{2}\sum_j\|\curl f_j\|_\lambda^2 = \eta
\end{equation}

\subsection{The Karman-Howarth-Monin Relations}\label{sec: prelim KHM eqs}
In this section we cover one of the key ideas behind the derivation of the third order flux laws: the Karman-Howarth-Monin (KHM) relations. These relations relate the various structure functions to two point correlations of the velocity or vorticity dissipation through the Navier-Stokes Equations. They were first derived for classical solutions to the determinstic Navier Stokes Equations by Karman and Howarth for the 4/5th law \eqref{eq: 4/5 law} in \cite{de1938statistical}. Later, they were generalized (for strong solutions) by Monin in \cite{monin2013statistical} and Eyink \cite{eyink1996exact} used an analogous KHM relation to derive Yaglom's law \eqref{eq: 4/3 law}. In the stochastic regime, Bedrossian et.al. generalized them for stationary martingale solutions on $\T^d$ for $d = 2,3$ \cites{bedrossian2019sufficient, bedrossian2020sufficient}.\newline

More specifically, for a stationary martingale solution to the Navier-Stokes equation and the vorticity equation (when $d=2$) we define the following vorticity and velocity structure functions
\begin{align*}
    S_{vor}(\ell) &= \E\fint_0^T\fint_{\T_\lambda^d}  |\delta_{\ell n}\omega^\nu|^2\delta_{\ell n}u^\nu \cdot n\;dxdtdS(n)\\
    S_{vel}(\ell) &= \E\fint_0^T\fint_{\T_\lambda^d} |\delta_{\ell n}u^\nu|^2\delta_{\ell n}u^\nu \cdot n\;dxdtdS(n)\\
    S_{vel}^\parallel(\ell) &= \E\fint_0^T\fint_{\T_\lambda^d} (\delta_{\ell n}u^\nu \cdot n)^3\;dxdtdS(n)
\end{align*}
where $dS(n)$ be the surface measure of the unit sphere $S^{d-1}$. We also define the the following two-point correlations:
\begin{align*}
    \Gamma_{vor}(\ell) &= \E\fint_0^T\fint_{S^{d-1}}\fint_{\T_\lambda^d} \omega^\nu \cdot T_{\ell n}\omega^\nu\;dxdS(n)dt\\
    \Gamma_{vel}(\ell) &= \E\fint_0^T\fint_{S^{d-1}}\fint_{\T_\lambda^d} u^\nu \cdot T_{\ell n}u^\nu\;dxdS(n)dt\\
    \Gamma_{vel}^\parallel(\ell) &= \E\fint_0^T\fint_{S^{d-1}}\fint_{\T_\lambda^d} (n \otimes n) : u^\nu \otimes T_{\ell n}u^\nu\;dxdS(n)dt\\
    a_{vor}(\ell) &= \frac{1}{2}\sum_j \E\fint_{S^{d-1}} \fint_{\T_\lambda^d} \curl f_j : T_{\ell n}\curl f_j\;dxdS(n)\\
    a_{vel}(\ell) &= \frac{1}{2}\sum_j \E\fint_{S^{d-1}} \fint_{\T_\lambda^d} f_j \cdot T_{\ell n} f_j\;dxdS(n)\\
    a_{vel}^\parallel(\ell) &= \frac{1}{2}\sum_j \E\fint_{S^{d-1}}\fint_{\T_\lambda^d} (n \otimes n) : f_j \otimes T_{\ell n} f_j\;dxdS(n)
\end{align*}
where $T_yg(x) = g(x+y)$ for all $y \in \R^d$. 

\begin{proposition}\label{prop: KHM relatons and regularity}
    Suppose $u^\nu$ is a solution to \eqref{eq: incompressible NSE} in the sense of Definition \ref{def: stationary martingale sol} on $\T_\lambda^d$ for $d = 2,3$. Then 
    \begin{align}
        S_{vel}(\ell) &= -4\nu\Gamma_{vel}'(\ell) - \frac{4}{\ell^{d-1}}\int_0^\ell r^{d-1}a_{vel}(r)\;dr\label{eq: velocity trans KHM}\\
        S_{vel}^\parallel(\ell) &= -4\nu(\Gamma_{vel}^\parallel)'(\ell) + \frac{2}{\ell^{d+1}}\int_0^\ell r^{d}S_{vel}(r)\;dr - \frac{4}{\ell^{d+1}}\int_0^\ell r^{d+1}a_{vel}^\parallel(r)\;dr\label{eq: velocity long KHM}
    \end{align}
    When $d = 2$, if $\omega^\nu$ is a solution to \eqref{eq: vorticity eq} in the sense of Definition \ref{def: stationary vorticity sol} then there is the additional KHM relation:
    \begin{equation}\label{eq: vorticity KHM}
        S_{vor}(\ell) = -4\nu\Gamma_{vor}'(\ell) - \frac{4}{\ell}\int_0^\ell ra_{vor}(r)\;dr
    \end{equation}
    Moreover, in $\R^3$ the correlations $\Gamma_{vel}, \Gamma_{vel}^\parallel \in C^2(\R^3)$ while in $\R^2$ the correlations $\Gamma_{vor} \in C^3(\R^2)$ and $\Gamma_{vel}, \Gamma_{vel}^\parallel \in C^4(\R^2)$.
\end{proposition}
\begin{remark}
    See \cite{bedrossian2019sufficient} for details about the derivation and regularity of the identities in 3D. In $\R^2$, the derivation of the structure functions is similar to the 3D case, but there is an increase in regularity due to boundedness of the average enstrophy. See \cite{bedrossian2020sufficient} for more details. 
\end{remark}

\subsection{Key Lemmas}\label{sec: prelim useful lemmas}
In this section, several commonly used calculations are computed. 
\subsubsection{Expression for isotropic tensors}
\begin{lemma}\label{lemma: isotropic tensor identity}
    Let $k \in \N$ and $\{i_j\}_{j=1}^{2k} \subset \{1, 2, \dots, d\}$. We define $\mathcal{S}_k$ to be the space of all pairwise combinations of the elements $\{i_j\}$ and let $n \in S^{d-1}$ for $d = 2,3$ and $dS(n)$ be the surface measure of the unit sphere $S^{d-1}$. If $\sigma(p)$ is the $p$th particular set of pairings of $\{i_j\}$ (i.e. $\sigma(p) \in \mathcal{S}_k$, with $1 \leq p \leq |\mathcal{S}_k|$), then the following identity holds:
    \begin{equation}\label{eq: isotropic tensors}
        \fint_{S^{d-1}} \prod_{j=1}^{2k} n_{i_j}\;dS(n) = \beta_d(k)\sum_{j=1}^{|\mathcal{S}_k|}\prod_{(a,b) \in \sigma(j)}\delta_{a,b}
    \end{equation}
    where 
    \[
        \beta_d(k) = \begin{cases}
        \frac{1}{2^kk!} & d = 2\\
        \frac{2^kk!}{(2k+1)!} & d = 3
        \end{cases}
    \]
\end{lemma}
\begin{proof}
    Due to the odd symmetry of each of the components $n_i$ over the unit sphere $S^{d-1}$, the integral
    \[
        \fint_{S^{d-1}} \prod_{j=1}^{2k} n_{i_j}\;dS(n) > 0
    \]
    if and only if there is an even number of each of the components. Moreover from trigonometric identities, one can deduce that it is equivalent to a linear combination of kronecker delta tensors symbolizing all of the possible pairings of the elements in $\{i_j\}$. Moreover as the ordering of the $\{i_j\}$ are interchangeable they must all occur with the same factor. Therefore,
    \[
        \fint_{S^{d-1}} \prod_{j=1}^{2k} n_{i_j}\;dS(n) = \beta_d(k)\sum_{j=1}^{|\mathcal{S}_k|}\prod_{(a,b) \in \sigma(j)}\delta_{a,b}
    \]
    where $\mathcal{S}_k$ is the collection of all pairwise combinations of $\{i_j\}$ and $\sigma(p) \in \mathcal{S}_k$ is the $p$th collection of pairings. For example when $k = 2$ we have 
    \[
        \mathcal{S}_2 = \big\{\underbrace{\{(i_1,i_2), (i_3, i_4)\}}_{\sigma(1)}, \underbrace{\{(i_1,i_3), (i_2,i_4)\}}_{\sigma(2)}, \underbrace{\{(i_1,i_4), (i_2, i_3)\}}_{\sigma(3)}\big\}
    \]
    It is a well known combinatorial result that the number of pairings is $|\mathcal{S}_k| = \frac{(2k)!}{2^kk!}$. 

    Now all that is left is to compute the constant $\beta_d(k).$ As all possible pairings have the same factor we consider the case with $i_1 = i_2 = \dots = i_{2k}$. As none of the terms on the left hand side vanish and we get
    \[
        |\mathcal{S}_k|\beta_d(k) = \fint_{S^{d-1}}n_{i_1}^{2k}\;dS(n).
    \]
    The right hand side depends on the dimension so we consider each case separately. When $d = 2$ we use the reduction formula to deduce:
    \[
        \fint_{S^1}(n_{i_1})^{2k}\;dS(n) = \frac{1}{2\pi}\int_0^{2\pi} \cos^{2k}\theta\;d\theta = \frac{(2k-1)!!}{(2k)!!} = \frac{(2k-1)!}{(2k-2)!!(2k)!!} = \frac{2(2k-1)!}{4^k(k-1)!k!}
    \]
    where we have used the identity $(2k)!! = \prod_{j=1}^k 2j = 2^kk!$. Thus 
    \[
        \beta_2(k) = \frac{2(2k-1)!}{4^k(k-1)!k!}\frac{1}{|\mathcal{S}_k|} = \frac{1}{2^kk!}
    \]
    Whereas when $d = 3$:
    \[
        \fint_{S^2} (n_{i_1})^{2k}\;dS(n) = \frac{1}{4\pi}\int_0^{2\pi}\int_0^\pi \cos^{2k}\phi \sin\phi\;d\phi d\theta = \frac{1}{2k+1}
    \]
    so 
    \[
        \beta_3(k) = \frac{1}{2k+1}\frac{1}{|\mathcal{S}_k|} = \frac{2^kk!}{(2k+1)!}
    \]
\end{proof}

Two direct consequences of Lemma \ref{lemma: isotropic tensor identity} are the following lemmas:
\begin{lemma}\label{lemma: tangential term integral}
    Let $r \in \R$, $p \in \N \cup \{0\}$, and $n,k \in \R^d$ such that $|n| = 1$. Then
    \[
        \Re\Big(\fint_{S^{d-1}} (n \cdot k)^{2p}e^{-ir(n \cdot k)}\;dS(n)\Big) = 
        |k|^{2p}\sum_{m=0}^\infty \beta_d(m+p)\frac{(-1)^m (r|k|)^{2m}(2m+2p)!}{2^{m+p}(m+p)!(2m)!}
    \]
\end{lemma}
\begin{proof}
    We use the power series expansion of $e^x$ and the fact that it is an entire function to consider the integration term-wise:
    \begin{align*}
        \Re\Big(\fint_{S^{d-1}} (n \cdot k)^{2p}e^{-ir(n \cdot k)}\;dS(n)\Big) &= \Re\Big(\sum_{m=0}^\infty \frac{(-ir)^m}{m!}\fint_{S^{d-1}} (n \cdot k)^{m+2p}\;dS(n)\Big)\\
        &= \sum_{m=0}^\infty \frac{(-1)^m r^{2m}}{(2m)!}\fint_{S^{d-1}} \prod_{j=1}^{2(m+p)}k_{i_j}n_{i_j}\;dS(n)
    \end{align*}
    where $k_{i_j}n_{i_j}$ is the $j$th copy of the inner product $k_in_i = k\cdot n$. Then it follows from Lemma \ref{lemma: isotropic tensor identity} that 
    \begin{align*}
        \fint_{S^{d-1}} \prod_{j=1}^{2(m+p)}k_{i_j}n_{i_j}\;dS(n) &= \beta_d(m+p)\prod_{j=1}^{2(m+p)}k_{i_j} \sum_{q=1}^{|\mathcal{S}_{m+p}|} \prod_{(a,b) \in \sigma(q)} \delta_{ab}\\
        &= \beta_d(m+p)\sum_{q=1}^{|\mathcal{S}_{m+p}|} |k|^{2m+2p}\\
        &= \beta_d(m+p)|\mathcal{S}_{m+p}||k|^{2m+2p}\\
        &= \beta_d(m+p)\frac{(2m+2p)!}{2^{m+p}(m+p)!}|k|^{2m+2p}
    \end{align*}
    Now we substitute this back into the sum to conclude the proof.
\end{proof}
\begin{lemma}\label{lemma: longitudinal term integral}
    Let $u$ be a divergence free random vector field and let $p \in \N\cup\{0\}$ and $r \in \R$, then
    \[
        \Re\Big(\fint_{S^{d-1}} n_an_b(n\cdot k)^{2p} e^{-irn\cdot k}\;dS(n)\E\widehat{u_a}(k)\widehat{u_b}(k)\Big) = |k|^{2p}\sum_{m=0}^\infty \beta_d(m+p+1)\frac{(-1)^m (r|k|)^{2m}(2m+2p)!}{2^{m+p}(m+p)!(2m)!}\E|\widehat{u}|^2
    \]
\end{lemma}
\begin{proof}
    Similar to the proof of Lemma \ref{lemma: tangential term integral}, we use the power series expansion of $e^x$ and the fact that it is an entire function to consider the integration term-wise:
     \begin{align*}
       \Re\Big(\fint_{S^{d-1}} n_an_b(n\cdot k)^{2p} e^{-irn\cdot k}\;dS(n)\E\widehat{u_a}\widehat{u_b}\Big) &= \Re\Big(\sum_{m=0}^\infty \frac{(-ir)^m}{m!}\fint_{S^{d-1}} n_an_b(n \cdot k)^{m+2p}\;dS(n)\E\widehat{u_a}\widehat{u_b}\Big)\\
        &= \sum_{m=0}^\infty \frac{(-1)^m r^{2m}}{(2m)!}\fint_{S^{d-1}} n_an_b\prod_{j=1}^{2(m+p)}k_{i_j}n_{i_j}\;dS(n)\E\widehat{u_a}\widehat{u_b}
    \end{align*}
    Note that on the Fourier side, the divergence free condition can be written as $k_j\widehat{u_j} = 0$. Now, as there are $2n+2p+2$ $n_j$ when we apply Lemma \ref{lemma: isotropic tensor identity} we get
    \begin{align*}
        \fint_{S^{d-1}} n_an_b\prod_{j=1}^{2(m+p)}k_{i_j}n_{i_j}\;dS(n)\E\widehat{u_a}\widehat{u_b} &= \beta_d(m+p+1)\prod_{j=1}^{2(m+p)}k_{i_j} \Big(\sum_{q=1}^{|\mathcal{S}_{m+p}|} \delta_{ab}\prod_{(a',b') \in \sigma(q)} \delta_{a'b'}\E\widehat{u_a}\widehat{u_b}\\
        &\quad + \sum_{q=1}^{|\mathcal{S}_{m+p}|} \sum_{(a',b') \in \sigma(q)}(\delta_{ab'}\delta_{a'b} + \delta_{a'a}\delta_{b'b})\prod_{\substack{(c',d') \in \sigma(q)\\
        (c',d') \neq (a',b')}}\delta_{c'd'}\E\widehat{u_a}\widehat{u_b}\Big)\\
        &= \beta_d(m+p+1) \sum_{q=1}^{|\mathcal{S}_{m+p}|}  |k|^{2m+2p}\E|\widehat{u}|^2\\
        &\quad + 2\beta_d(m+p+1)\sum_{q=1}^{|\mathcal{S}_{m+p}|} |k|^{2m+2p-2}\E k_a\widehat{u}_ak_b\widehat{u}_b\\
        &= \beta_d(m+p+1)|\mathcal{S}_{m+p}||k|^{2m+2p}\E|\widehat{u}|^2\\
        &= \beta_d(m+p+1)\frac{(2m+2p)!}{2^{m+p}(m+p)!} |k|^{2m+2p}\E|\widehat{u}|^2
    \end{align*}
    Substituting this back into the sum concludes the proof.
\end{proof}

\subsubsection{Filtration Limits}
In this section, we will present the two main lemmas which will provide the characterization of the direct and inverse cascades. First we consider the case when small scale limit i.e. when $\ell \to 0$ and characterize the existence of a cascade as when there is a sequence of wave-numbers for the energy to ``escape out to infinity".  
\begin{lemma}[Small Scale Limit]\label{lemma: small scale wave-number sum limit}
    Let $c(\ell,k)$ be a bounded sequence of real numbers such that $\ds\lim_{\ell|k| \to 0}c(\ell,k) = 0$ and $\ds\lim_{\ell|k| \to \infty} c(\ell,k) = L$ with $L \neq 0$. Let $f^\nu$ is a sequence of $L^2$ random variables such that 
    \[
        \limsup_{\nu \to 0}\E\|f^\nu\|_\lambda^2 < \infty
    \]
    Then there exists $\ell_\nu \to 0$ such that  
    \[
        \lim_{\ell_I \to 0}\liminf_{\nu \to 0}\inf_{\ell \in (\ell_\nu, \ell_I)}\sum_{k}c(\ell,k)\E|\widehat{f^\nu}(k)|^2 = L\delta
    \]
    if and only if there exists $N_\nu \to \infty$ such that 
    \[
        \liminf_{\nu \to 0}\sum_{|k| \geq N_\nu}\E|\widehat{f^\nu}|^2 = \delta
    \]
\end{lemma}
\begin{proof}
    $(\Rightarrow)$ Suppose there exists $\ell_\nu \to 0$ such that 
    \[
        \lim_{\ell_I \to 0}\liminf_{\nu \to 0}\inf_{\ell \in (\ell_\nu, \ell_I)}\sum_{k}c(\ell,k)\E|\widehat{f^\nu}(k)|^2 = L\delta
    \]
    Pick $N_\nu \to \infty$ such that $\ds\lim_{\nu \to 0}\ell_\nu N_\nu = \infty$ and consider the sums
    \[
        \sum_{k}c(\ell,k)\E|\widehat{f^\nu}(k)|^2  = \sum_{|k| < N_\nu}c(\ell,k)\E|\widehat{f^\nu}(k)|^2 + \sum_{|k| \geq N_\nu}c(\ell,k)\E|\widehat{f^\nu}(k)|^2  =: I_1 + I_2
    \]
    First consider $I_2$ and observe
    \[
        \inf_{|k| \geq N_\nu} c(\ell,k)\sum_{|k| \geq N_\nu}\E|\widehat{f^\nu}(k)|^2 \leq I_2 \leq \sup_{|k|\geq N_\nu}\sum_{|k| \geq N_\nu}c(\ell,k)\E|\widehat{f^\nu}(k)|^2
    \]
    Apply the inf in $\ell$ to each side of the inequality:
    \[
        \inf_{\ell|k| \geq \ell_\nu N_\nu}c(\ell,k)\sum_{|k| \geq N_\nu}\E|\widehat{f^\nu}(k)|^2 \leq \inf_{\ell \in (\ell_\nu, \ell_I)}I_2 \leq \sup_{\ell|k| \geq \ell_\nu N_\nu}c(\ell,k)\sum_{|k| \geq N_\nu}\E|\widehat{f^\nu}(k)|^2 
    \]
    Next we take the limits in $\nu$ and $\ell_I$. Note that by our choice of $N_\nu$ and the fact that $\ds\lim_{\ell|k| \to \infty}c(\ell,k) = L$ it follows that 
    \begin{align*}
        \lim_{\ell_I \to 0}\liminf_{\nu \to 0}\inf_{\ell|k| \geq \ell_\nu N_\nu}c(\ell,k)\sum_{|k| \geq N_\nu}\E|\widehat{f^\nu}(k)|^2 &= L\liminf_{\nu \to 0}\sum_{|k| \geq N_\nu}\E|\widehat{f^\nu}(k)|^2\\
        \lim_{\ell_I \to 0}\liminf_{\nu \to 0}\sup_{\ell|k| \geq \ell_\nu N_\nu}c(\ell,k)\sum_{|k| \geq N_\nu}\E|\widehat{f^\nu}(k)|^2 &= L\liminf_{\nu \to 0}\sum_{|k| \geq N_\nu}\E|\widehat{f^\nu}(k)|^2
    \end{align*}
    Hence by the Squeeze theorem 
    \[
        \lim_{\ell_I \to 0}\liminf_{\nu \to 0}\inf_{\ell|k| \geq \ell_\nu N_\nu}I_2 = L\liminf_{\nu \to 0}\sum_{|k| \geq N_\nu}\E|\widehat{f^\nu}(k)|^2
    \]
    Next we claim that 
    \[
        \lim_{\ell_I \to 0}\liminf_{\nu \to 0}\inf_{\ell \in (\ell_\nu, \ell_I)}|I_1| = 0
    \]
    For the sake of contradiction suppose that
    \[
        \lim_{\ell_I \to 0}\liminf_{\nu \to 0}\inf_{\ell \in (\ell_\nu, \ell_I)}|I_1| > 0
    \]
    Then there exists $|k| < \infty$ and $\varepsilon_0$ such that for all $\nu > 0, \ell_I > 0$ there exists $\nu^* < \nu$ and $\ell_I^* < \ell_I$ 
    \[
        \varepsilon_0 < \inf_{\ell \in (0, \ell_I^*)}|c(\ell,k)|\E|\widehat{f^{\nu^*}}(k)|^2 \leq C\sup_{\ell \in 0, \ell_I)}|c(\ell,k)|
    \]
    However, as $\ds\lim_{\ell|k| \to 0}c(\ell,k) = 0$, we can find $\ell_I$ small enough such that 
    \[
        \sup_{\ell \in( 0, \ell_I)}|c(\ell,k)| < \varepsilon_0/C
    \]
    a contradiction. Thus $|I_1| \to 0$ as $\ell_I \to 0$ and $\nu \to 0$.     
    Putting it all together now we get
    \begin{align*}
        \liminf_{\nu \to 0}\sum_{|k| \geq N_\nu}\E|\widehat{f^\nu}(k)|^2 = \frac{1}{L}\lim_{\ell_I \to 0}\liminf_{\nu \to 0}\inf_{\ell \in (\ell_\nu, \ell_I)}\sum_{k}c(\ell,k)\E|\widehat{f^\nu}(k)|^2 = \delta
    \end{align*}
    ($\Leftarrow$) Now we assume that there exists $N_\nu \to \infty$ such that 
    \[
        \liminf_{\nu \to 0}\sum_{|k| \geq N_\nu}\E|\widehat{f^\nu}|^2 = \delta
    \]
    Pick $\ell_\nu \to 0$ such that $\ds\lim_{\nu \to 0}\ell_\nu N_\nu = \infty$. Repeat the steps above to conclude that 
    \begin{align*}
        \lim_{\ell_I \to 0}\liminf_{\nu \to 0}\inf_{\ell \in (\ell_\nu, \ell_I)} I_1 = 0\\
        \lim_{\ell_I \to 0}\liminf_{\nu \to 0}\inf_{\ell \in (\ell_\nu, \ell_I)}I_2 = L\delta
    \end{align*}
\end{proof}

Notice that the small scale limit filters out everything except what happens to the random variables at infinite wave-numbers. As such we will sometimes refer to Lemma \ref{lemma: small scale wave-number sum limit} as the high-frequency filtration limit. Similarly we can characterize the large scale limit as filtering out everything except what happens at low frequencies. 

\begin{lemma}[Large Scale Limit]\label{lemma: large scale wave-number sum limit}
    Let $c(\ell,k)$ be a sequence of real numbers such that $\ds\lim_{\ell|k| \to 0}c(\ell,k) = 0$ and $\ds\lim_{\ell|k| \to \infty} c(\ell,k) = L$ with $L \neq 0$. Let $f^\nu$ is a sequence of $L^2$ random variables such that 
    \[
        \limsup_{\nu \to 0}\E\|f^\nu\|_\lambda^2 = \Delta < \infty.
    \]
    Then there exists $\tilde{\ell}_\nu \to \infty$ such that  
   \[
        \lim_{\ell_I \to \infty}\limsup_{\nu \to 0}\sup_{\ell \in (\ell_I, \tilde{\ell}_\nu)}\sum_{k}c(\ell,k)\E|\widehat{f^\nu}(k)|^2 = L(\Delta - \delta)
    \]
    if and only if there exists $M_\nu \to 0$ such that 
    \[
        \liminf_{\nu \to 0}\sum_{|k| \leq M_\nu}\E|\widehat{f^\nu}|^2 = \delta
    \]
\end{lemma}
\begin{proof}
    The proof is very similar to the proof of Lemma \ref{lemma: small scale wave-number sum limit}. \newline
    ($\Rightarrow$) Suppose there exists $\tilde{\ell}_\nu$ such that
    \[
        \lim_{\ell_I \to \infty}\limsup_{\nu \to 0}\sup_{\ell \in (\ell_I, \tilde{\ell}_\nu)}\sum_{k}c(\ell,k)\E|\widehat{f^\nu}(k)|^2 = L(\Delta - \delta)
    \]
    Pick $M_\nu \to 0$ such that $M_\nu = O(\tilde{\ell}_\nu^{-1})$. By assumption $c(\ell,0) = 0$ for all $\ell$.
    Hence 
    \begin{align*}
         \sum_{k}c(\ell,k)\E|\widehat{f^\nu}(k)|^2 &= \sum_{|k| \geq 1}c(\ell,k)\E|\widehat{f^\nu}(k)|^2\\
         &= \sum_{1 \leq |k|\leq M_\nu}c(\ell,k)\E|\widehat{f^\nu}(k)|^2 + \sum_{|k| \geq \max\{M_\nu ,1\}}c(\ell,k)\E|\widehat{f^\nu}(k)|^2\\
         &=: I_1 + I_2
    \end{align*}
    As the sequence $c(\ell,k)$ is convergent, it is also bounded uniformly in $\ell$ and $k$. Moreover, from our choice of $M_\nu$ it follows that  
    \[
        \lim_{\ell_I \to \infty}\limsup_{\nu \to 0}\sup_{\ell \in (\ell_I, \tilde{\ell}_\nu)} |I_1| \leq C\limsup_{\nu \to 0} \sum_{1\leq |k| \leq M_\nu}\E|\widehat{f^\nu}|^2 = 0
    \]
    Next consider the term $I_2$. Define $M_\nu^* = \max\{M_\nu, 1\}$ then 
    \[
        \inf_{|k| \geq M_\nu^*}c(\ell,k) \sum_{|k| \geq M_\nu^*} \E|\widehat{f^\nu}|^2 \leq I_2 \leq \sup_{|k| \geq M_\nu^*}c(\ell,k) \sum_{|k| \geq M_\nu^*}\E|\widehat{f^\nu}|^2 
    \]
    Apply the sup in $\ell$ to each side of the inequality:
    \[
        \inf_{|k|\ell \geq \ell_I M_\nu^*}c(\ell,k) \sum_{|k| \geq M_\nu^*} \E|\widehat{f^\nu}|^2 \leq \sup_{\ell \in (\ell_I, \tilde{\ell}_I)}I_2 \leq \sup_{|k|\ell \geq \ell_I M_\nu^*}c(\ell,k)\sum_{|k| \geq M_\nu^*} \E|\widehat{f^\nu}|^2 
    \]
    Finally take the limit in $\nu$ and $\ell_I$ and use the fact that $\ds\lim_{\ell|k| \to \infty}c(\ell,k) = L$:
    \begin{align*}
        \lim_{\ell_I \to \infty}\limsup_{\nu \to 0}\inf_{\ell|k| \geq \ell_IM_\nu^*}c(\ell,k)\sum_{|k| \geq M_\nu^*}\E|\widehat{f^\nu}(k)|^2 &= \lim_{\ell_I \to \infty}\inf_{\ell|k| \geq \ell_I}c(\ell,k)\limsup_{\nu \to 0}\sum_{|k| \geq M_\nu^*}\E|\widehat{f^\nu}(k)|^2\\
        &= L\limsup_{\nu \to 0}\sum_{|k| \geq M_\nu^*}\E|\widehat{f^\nu}(k)|^2\\
        \lim_{\ell_I \to \infty}\limsup_{\nu \to 0}\sup_{\ell|k| \geq \ell_IM_\nu^*}c(\ell,k)\sum_{|k| \geq M_\nu^*}\E|\widehat{f^\nu}(k)|^2 &= \lim_{\ell_I \to \infty}\sup_{\ell|k| \geq \ell_I}c(\ell,k)\limsup_{\nu \to 0}\sum_{|k| \geq M_\nu^*}\E|\widehat{f^\nu}(k)|^2\\
        &= L\limsup_{\nu \to 0}\sum_{|k| \geq M_\nu^*}\E|\widehat{f^\nu}(k)|^2
    \end{align*}
    Hence by the Squeeze theorem 
    \[
        \lim_{\ell_I \to \infty}\limsup_{\nu \to 0}\sup_{\ell \in (\ell_I, \tilde{\ell}_\nu}I_2 = L\limsup_{\nu \to 0}\sum_{|k| \geq M_\nu^*}\E|\widehat{f^\nu}(k)|^2
    \]
    Putting it all together results in 
    \begin{align*}
        \liminf_{\nu \to 0} \sum_{|k| \leq M_\nu}\E|\widehat{f^\nu}(k)|^2 &=  \Delta - \limsup_{\nu \to 0}\sum_{|k| \geq M_\nu^*}\E|\widehat{f^\nu}(k)|^2\\
        &= \Delta - \frac{1}{L}\lim_{\ell_I \to \infty}\limsup_{\nu \to 0}\sup_{\ell \in (\ell_I, \tilde{\ell}_\nu}\sum_{k}c(\ell,k)\E|\widehat{f^\nu}(k)|^2\\
        &= \delta
    \end{align*}
    ($\Leftarrow$) Now assume that there exists $M_\nu \to 0$ such that 
    \[
        \delta = \liminf_{\nu \to 0}\sum_{|k| \leq M_\nu}\E|\widehat{f^\nu}|^2 = \Delta - \limsup_{\nu \to 0}\sum_{|k| \geq \max\{M_\nu, 1\}}\E|\widehat{f^\nu}|^2
    \]
    Simplifying, we have $\ds\limsup_{\nu \to 0}\sum_{|k| \geq \max\{M_\nu, 1\}}\E|\widehat{f^\nu}|^2 = \Delta - \delta$.
    
    Choose $\tilde{\ell}_\nu = O(M_\nu^{-1})$ and repeat the steps above to conclude that 
    \begin{align*}
        \lim_{\ell_I \to 0}\limsup_{\nu \to 0}\sup_{\ell \in (\ell_\nu, \ell_I)} I_1 = 0\\
        \lim_{\ell_I \to 0}\limsup_{\nu \to 0}\sup_{\ell \in (\ell_\nu, \ell_I)} I_2 = L(\Delta - \delta)        
    \end{align*}
\end{proof}

\section{The 3D Direct Cascade}\label{sec: 3d direct cascade}
In this section we establish the characterization of the direct cascade flux laws in 3D.
\begin{theorem}[3D Direct Cascade Characterization]
    Suppose $\{u\}_{\nu > 0}$ is a sequence of statistically stationary solutions.
    There exists $N_\nu \geq 1$ such that $\ds\lim_{\nu \to 0}N_\nu = \infty$ and 
    \[
        \liminf_{\nu \to 0}\Big(\nu \sum_{|k| \geq N_\nu}\E|\widehat{\grad u}|^2 + \E D(u^\nu)\Big) = \varepsilon^*
    \]
    if and only if there exists $\ell_\nu \in (0, 1)$ such that $\ds\lim_{\nu \to 0}\ell_\nu = 0$ and 
    \begin{align}
        \lim_{\ell_I \to 0}\limsup_{\nu \to 0}\sup_{\ell \in [\ell_\nu, \ell_I]}\Big|\frac{S_{vel}(\ell)}{\ell^3} + \frac{4}{3}\varepsilon^*\Big| = 0\label{eq: 3d S0 limit}\\
        \lim_{\ell_I \to 0}\limsup_{\nu \to 0}\sup_{\ell \in [\ell_\nu, \ell_I]}\Big|\frac{S_{vel}^\parallel(\ell)}{\ell} + \frac{4}{5}\varepsilon^*\Big| = 0\label{eq: 3d S long limit}
    \end{align}
\end{theorem}
\subsection{Proof of \eqref{eq: 3d S0 limit}}
First, recall the KHM equation for $S_{vel}$ in three dimensions:
\[
    \frac{S_{vel}}{\ell} = -\frac{4\nu}{\ell}\Gamma_{vel}'(\ell) - \frac{4}{\ell^3}\int_0^\ell r^2a_{vel}(r)\;dr
\]
\textbf{Step 1:} Beginning with the forcing term, we show that 
    \[
         \frac{-4}{\ell^3}\int_0^\ell r^2a_{vel}(r)\;dr = -\frac{4}{3}\varepsilon + o_{\ell \to 0}(1)
    \]
    Note that 
    \[
        a_{vel}(0) = \frac{1}{2}\sum_{j}\E\fint_{\T_\lambda^3} |f_j(x)|^2\;dx = \varepsilon.
    \]
    Furthermore, as the $f_j$ are smooth, we Taylor expand around $r = 0$ using the Peano formulation of the error (as we do not need a rigorous analysis of the error in this case).
    \begin{align*}
        \frac{-4}{\ell^3}\int_0^\ell r^2a_{vel}(r)\;dr &= \frac{-2}{\ell^3}\int_0^\ell r^2\E\fint_{S^2}\sum_{j}\fint_{\T_\lambda^3} f_j\cdot T_{rn}f_j\;dxdS(n)dr\\
        &= \frac{-2}{\ell^3}\int_0^\ell r^2\E\fint_{S^2}\sum_{j}\fint_{\T_\lambda^3} |f_j|^2\;dxdS(n) + r^2h_0(r)dr\\
        &= \frac{-2}{3}\sum_{j}\fint_{\T_\lambda^3} \E|f_j|^2\;dx + \frac{-2}{\ell^3}\int_0^\ell r^2h_0(r)\;dr\\
        &= \frac{-4}{3}\varepsilon + \frac{-2}{\ell^3}\int_0^\ell r^2h_0(r)\;dr
    \end{align*}
    where $h_0(r)$ is a function such that $\ds\lim_{r \to 0}h_0(r) = 0$. Hence the error disappears as $\ell \to 0$ or in other words $\frac{-2}{\ell^3}\int_0^\ell r^2h_0(r)\;dr = o_{\ell \to 0}(1)$.
    
\textbf{Step 2:} Next we show that 
    \[
        \frac{-4\nu}{\ell}\Gamma_{vel}'(\ell) = \frac{4}{3}\varepsilon - \frac{4}{3}\E D(u^\nu) - \sum_k c(k, \ell)\E\fint_0^T|\widehat{\grad u^\nu}|^2
    \]
    where $\ds\lim_{\ell|k| \to 0} c(\ell,k) = 0$ and $\ds\lim_{\ell|k| \to \infty} c(\ell,k) = \frac{4}{3}$. 
    
    From Proposition \ref{prop: KHM relatons and regularity} it is known that $\Gamma_{vel} \in C^2$, so if we Taylor expand $\Gamma_{vel}'$ about 0 and use the Lagrange formulation of the error to get:
    \[
        \frac{-4\nu}{\ell}\Gamma_{vel}'(\ell) = \frac{-4\nu}{\ell}\Gamma_{vel}'(0) + \frac{-4\nu}{\ell}\int_0^\ell \Gamma_{vel}''(r)\;dr 
    \]
    By integration by parts 
    \begin{align*}
         \Gamma_{vel}'(0) &= -\E\fint_{S^2}\fint_{\T_\lambda^2} n_iu_m^\nu\partial_iu_m^\nu\;dxdS(n) = 0
    \end{align*}
    Next, consider the error term, which we will note is real valued. We use Plancheral's theorem to rewrite the $L^2$ inner product in $x$ into a Fourier series. Then we apply Fubini's theorem and Lemma \ref{lemma: tangential term integral} when $p = 1$ to get
    \begin{align*}
        \frac{-4\nu}{\ell}\int_0^\ell \Gamma_{vel}''(r)\;dr &= \frac{4\nu}{\ell}\int_0^\ell \E\fint_{S^2}\fint_{\T_\lambda^3}n_in_j\partial_iu_m^\nu\partial_jT_{rn}u_m^\nu\;dxdS(n)dr\\
        &= \frac{-4\nu}{\ell} \sum_k\int_0^\ell \fint_{S^2} (n \cdot k)^2e^{-irn\cdot k}\E|\widehat{u^\nu}(k)|^2\;dS(n)dr\\
        &= \frac{-4\nu}{\ell}\sum_k\int_0^\ell |k|^2\sum_{m=0}^\infty \frac{(-1)^m(r|k|)^{2m}}{(2m)!(2m+3)}\;dr\E|\widehat{u^\nu}(k)|^2\\
        &= 4\nu\sum_k\sum_{m=0}^\infty \frac{(-1)^m(\ell|k|)^{2m}}{(2m+1)!(2m+3)}\E\fint_0^T|\widehat{\grad u^\nu}(k)|^2\\
        &= \frac{4\nu}{3}\sum_k \E\fint_0^T|\widehat{\grad u^\nu}|^2 - 4\nu\sum_k\sum_{m=1}^\infty \frac{(-1)^{m+1}(\ell|k|)^{2m}}{(2m+1)!(2m+3)}\E\fint_0^T|\widehat{\grad u^\nu}(k)|^2\\
        &= \frac{4}{3}\varepsilon - \frac{4}{3}\E D(u^\nu) - \sum_{k}c(\ell,k)\E\fint_0^T|\widehat{\grad u^\nu}|^2
    \end{align*}
    Notice that when applying Lemma \ref{lemma: tangential term integral} we substituted in for the value of $\beta_3(m+1)$ and simplified the coefficient in the power series in $m$. Also the ultimate equality comes from the energy balance \eqref{eq: energy balance} to write the first term as $\frac{4}{3}\varepsilon.$
    
    From the definition of $c(\ell,k)$ it is readily seen that $\ds\lim_{\ell|k| \to 0}c(\ell,k) = 0$.
    \begin{align*}
        \Big|\frac{4}{3} - c(\ell,k)\Big| &= \Big|4\sum_{m=0}^\infty \frac{(-1)^m(\ell|k|)^{2m}}{(2m+1)!(2m+3)}\Big|\\
        &= \Big|\frac{4}{(\ell|k|)^3}\sum_{m=0}^\infty \frac{(-1)^m(\ell|k|)^{2m+3}}{(2m+1)!(2m+3)}\Big|\\
        &= \Big|\frac{4}{(\ell|k|)^3}\sum_{m=0}^\infty \int_0^{\ell|k|} \frac{(-1)^m t^{2m+2}}{(2m+1)!}\Big|\\
        &= \Big|\frac{4}{(\ell|k|)^3} \int_0^{\ell|k|} t\sin t\;dt\Big|\\
        &\leq \frac{4}{(\ell|k|)^3}\int_0^{\ell|k|} t\;dt\\
        &= \frac{2}{\ell|k|}
    \end{align*}
    so $\ds\lim_{\ell|k| \to \infty}c(\ell,k) = \frac{4}{3}$.
        
    \textbf{Step 3:} It follows from Lemma \ref{lemma: small scale wave-number sum limit} that there exists $\ell_\nu \to 0$ such that 
    \begin{align*}
        \lim_{\ell_I \to 0}\limsup_{\nu \to 0}\sup_{\ell \in [\ell_\nu, \ell_I]} \frac{-4\nu}{\ell}\Gamma'(\ell) &= \frac{4}{3}\varepsilon - \frac{4}{3}\liminf_{\nu \to 0}\E D(u^\nu) - \lim_{\ell_I \to 0}\liminf_{\nu \to 0}\inf_{\ell \in [\ell_\nu, \ell_I]}\nu\sum_{k}c(\ell,k)\E|\widehat{\grad u}|^2\\
        &= \frac{4\varepsilon}{3} - \frac{4\varepsilon^*}{3}
    \end{align*}
    if and only if there exists $N_\nu \to \infty$ such that 
    \[
         \liminf_{\nu \to 0}\E D(u^\nu) + \liminf_{\nu \to 0}\nu\sum_{|k| \geq N_\nu}\E|\widehat{\grad u}|^2 = \varepsilon^*
    \]
    Finally, combine Steps 1-3 with the KHM equation to derive \eqref{eq: 3d S0 limit}.

\subsection{Proof of \eqref{eq: 3d S long limit}}
    Recall the KHM equation for $S_{vel}^\parallel$ on $\T_\lambda^3:$ 
    \[
        \frac{S_{vel}^\parallel(\ell)}{\ell} = -\frac{4\nu}{\ell}(\Gamma_{vel}^\parallel)'(\ell) + \frac{2}{\ell^5}\int_0^\ell r^3S_{vel}(r)\;dr - \frac{4}{\ell^5}\int_0^\ell r^4a_{vel}^\parallel(r)\;dr
    \]
    \textbf{Step 1:} Beginning with the forcing term, we show that 
    \[
         \frac{-4}{\ell^5}\int_0^\ell r^4a_{vel}^\parallel(r)\;dr = -\frac{-4}{15}\varepsilon + o_{\ell \to 0}(1)
    \]
    Note that 
    \[
        a_{vel}^\parallel(0) = \frac{1}{3}\sum_{j}\E\fint_{\T_\lambda^3} |f_j(x)|^2\;dx = \frac{\varepsilon}{3}.
    \]
    Furthermore, as the $f_j$ are smooth, we Taylor expand around $r = 0$ using the Peano formulation of the error (as we do not need a rigorous analysis of the error in this case).
    \begin{align*}
        \frac{-4}{\ell^5}\int_0^\ell r^4a_{vel}^\parallel(r)\;dr &= \frac{-4}{\ell^5}\int_0^\ell r^4a_{vel}^\parallel(0) + r^4h_0(r)\;dr\\
        &= \frac{-4}{15}\varepsilon - \frac{4}{\ell^5}\int_0^\ell r^4h_0(r)\;dr 
    \end{align*}
    where $h_0(r)$ is a function such that $\ds\lim_{r \to 0}h_0(r) = 0$. Hence by L'hopital's Theorem, the error disappears as $\ell \to 0$ uniformly with respect to $\nu$ or in other words $\frac{-4}{\ell^5}\int_0^\ell r^2h_0(r)\;dr = o_{\ell \to 0}(1)$.
    
    \textbf{Step 2:} Next we show that 
    \[
        \frac{-4\nu}{\ell}(\Gamma_{vel}^\parallel)'(\ell) = \frac{4}{15}\varepsilon - \sum_{k}c(\ell,k)\E|\widehat{\grad u}(k)|^2
    \]
    where the coefficients satisfy $\ds\lim_{\ell|k| \to 0}c(\ell,k) = 0$ and $\ds\lim_{\ell|k| \to \infty}c(\ell,k) = \frac{4}{15}$

    From Proposition \ref{prop: KHM relatons and regularity}, we know $\Gamma_{vel}^\parallel \in C^2$. We Taylor expand this about $r = 0$ and use the Lagrange formulation of the error to get:
    \[
        \frac{-4\nu}{\ell}(\Gamma_{vel}^\parallel)'(\ell) =  \frac{-4\nu}{\ell}(\Gamma_{vel}^\parallel)'(0) - \frac{4\nu}{\ell}\int_0^\ell (\Gamma_{vel}^\parallel)''(r)\;dr
    \]
    By integration by parts
    \begin{align*}
        \frac{-4\nu}{\ell}\Gamma_\parallel'(0) &= \frac{4\nu}{\ell}\E\fint_{S^2}\fint_{\T_\lambda^2}n_in_jn_k\partial_ku_i^\nu u_j^\nu\;dxdS(n) = 0
    \end{align*}
    Next, consider the error term, which we will note is real valued. We use Plancheral's theorem to rewrite the $L^2$ inner product in $x$ into a Fourier series. Then we apply Fubini's theorem and Lemma \ref{lemma: longitudinal term integral} when $p = 1$ to get 
    \begin{align*}
        \frac{-4\nu}{\ell}\int_0^\ell \Gamma_\parallel''(r)\;dr &= \frac{4\nu}{\ell}\int_0^\ell \E\fint_{S^2}\fint_{\T_\lambda^3}n_in_jn_pn_q\partial_iu_p^\nu\partial_jT_{rn}u_q^\nu\;dxdS(n)dr\\
        &= \frac{-4\nu}{\ell} \sum_k\int_0^\ell \fint_{S^2} n_pn_q(n \cdot k)^2e^{-irn\cdot k}\E\widehat{u_p^\nu}(k)\widehat{u_q^\nu}(k)\;dS(n)dr\\
        &= \frac{4\nu}{\ell} \sum_k \int_0^\ell \sum_{m=0}^\infty \frac{(-1)^m (r|k|)^{2m}}{(2m)!(2m+3)(2m+5)}\;dr\E\fint_0^T|\widehat{\grad u^\nu}|^2\\
        &= 4\nu \sum_k \sum_{m=0}^\infty \frac{(-1)^m (\ell|k|)^{2m}}{(2m+1)!(2m+3)(2m+5)}\E\fint_0^T|\widehat{\grad u^\nu}|^2\\
        &= \frac{4\nu}{15}\sum_{k}\E\fint_0^T|\widehat{\grad u^\nu}|^2 - 4\nu \sum_k \sum_{m=1}^\infty \frac{(-1)^{m+1} (\ell|k|)^{2m}}{(2m+1)!(2m+3)(2m+5)}\E\fint_0^T|\widehat{\grad u^\nu}|^2\\
        &= \frac{4\varepsilon}{15} - \frac{4}{15}\E D(u^\nu) - \sum_{k}c(\ell,k)\E\fint_0^T|\widehat{\grad u^\nu}|^2
    \end{align*}
    Note that when we applied Lemma \ref{lemma: longitudinal term integral}, we substituted in the definition of $\beta_3(m+2)$ and then simplified the coefficient. Then in the final equality we have applied the energy balance \eqref{eq: energy balance}. 
    From the definition of $c(\ell,k)$ it is readily seen that $\ds\lim_{\ell|k| \to 0}c(\ell,k) = 0$. Also
    \begin{align*}
        \Big|\frac{4}{15} - c(\ell,k)\Big| &= \Big|4\sum_{m=0}^\infty \frac{(-1)^m (\ell|k|)^{2m}}{(2m+1)!(2m+3)(2m+5)}\Big|\\       
        &=\Big|\frac{4}{(\ell|k|)^5}\sum_{m=0}^\infty \frac{(-1)^m(\ell|k|)^{2m+5}}{(2m+1)!(2m+3)(2m+5)}\Big|\\
        &= \Big|\frac{4}{(\ell|k|)^5}\int_0^{\ell|k|}\sum_{m=0}^\infty \frac{(-1)^m}{(2m+1)!(2m+3)} t^{2m+4}\;dt\Big|\\
        &= \Big|\frac{4}{(\ell|k|)^5}\int_0^{\ell|k|}t\int_0^t\sum_{m=0}^\infty \frac{(-1)^m}{(2m+1)!} s^{2m+2}\;dsdt\Big|\\
        &= \Big|\frac{4}{(\ell|k|)^5}\int_0^{\ell|k|}t\int_0^t s\sin(s)\;dsdt\Big|\\
        &\leq \frac{C}{(\ell|k|)^5}\int_0^{\ell|k|}t\int_0^ts\;dsdt\\
        &\leq \frac{C}{\ell|k|}
    \end{align*}
    so $\ds\lim_{\ell|k| \to \infty}c(\ell,k) = \frac{4}{15}$.

    \textbf{Step 3:} Apply Lemma \ref{lemma: small scale wave-number sum limit} to conclude that there exists $\ell_\nu \to 0$ such that 
    \begin{align*}
         \lim_{\ell_I \to 0}\limsup_{\nu \to 0}\sup_{\ell \in (\ell_\nu, \ell_I)}\frac{-4\nu}{\ell}\Gamma_{\parallel}'(\ell) &= \frac{4\varepsilon}{15} - \lim_{\ell_I \to 0}\liminf_{\nu \to 0}\inf_{\ell \in (\ell_\nu, \ell_I)}\sum_kc(\ell,k)\E|\widehat{\grad u}(k)|^2\\
         &= \frac{4}{15}\varepsilon - \frac{4}{15}\varepsilon^*
    \end{align*}
    if and only if there exists $N_\nu \to \infty$ such that 
    \[
        \liminf_{\nu \to 0} \E D(u^\nu) + \liminf_{\nu \to 0}\nu\sum_{|k| \geq N_\nu}\E|\widehat{\grad u}|^2 = \varepsilon^*.
    \]
    \textbf{Step 4:} Next we use Equation \eqref{eq: 3d S0 limit}, to conclude that the if and only if condition in step 3, also implies
    \[
        \frac{2}{\ell^5}\int_0^\ell r^3S_{vel}(r)\;dr = \frac{2}{\ell^5}\int_0^\ell r^4\frac{S_{vel}(r)}{r}\;dr = \frac{-8}{15}\varepsilon^* + o_{\ell \to 0}(1)  
    \]
    Finally combine Steps 1-4 with the KHM equation to derive \eqref{eq: 3d S long limit}.

\section{The 2D Direct Cascade}\label{sec: 2d direct cascade}
In this section, we characterize the existence of a direct cascade in 2D. Note that the proof is almost identical to the 3D case, but as some of the details are not trivial we fully show the derivation in 2D.
\begin{theorem}[2D Direct Cascade Characterization]
    Suppose $\{(u^\nu, \omega^\nu)\}_{\nu > 0}$ is a sequence of statistically stationary solutions satisfying both the Navier-Stokes equations \eqref{eq: incompressible NSE} and the vorticity equation \eqref{eq: vorticity eq}. Then there exists $N_\nu \geq 1$ such that $\ds\lim_{\nu \to 0}N_\nu = \infty$ and
    \[
       \liminf_{\nu \to 0}\nu\sum_{|k| \geq N_\nu}\E|\widehat{\grad \omega^\nu}|^2 = \eta^* 
    \]
    if and only if there exists $\ell_\nu \in (0,1)$ such that $\ds \lim_{\nu \to 0}\ell_\nu = 0$ and
    \begin{align}
        \lim_{\ell_I \to 0}\limsup_{\nu \to 0}\sup_{\ell \in [\ell_\nu, \ell_I]}\Big|\frac{S_{vor}(\ell)}{\ell} + 2\eta^*\Big| = 0\label{eq: vorticity limit}\\
        \lim_{\ell_I \to 0}\limsup_{\nu \to 0}\sup_{\ell \in [\ell_\nu, \ell_I]}\Big|\frac{S_{vel}(\ell)}{\ell^3} - \frac{1}{4}\eta^*\Big| = 0\label{eq: 2d S0 limit}\\
        \lim_{\ell_I \to 0}\limsup_{\nu \to 0}\sup_{\ell \in [\ell_\nu, \ell_I]}\Big|\frac{S_{vel}^\parallel(\ell)}{\ell^3} - \frac{1}{8}\eta^*\Big| = 0\label{eq: 2d S long limit}
    \end{align}
\end{theorem}

\subsection{Proof of \eqref{eq: vorticity limit}}\label{sec: 2d direct vorticity limit} 
    First, recall the KHM equation for $S_{vor}:$
    \[
        \frac{S_{vor}(\ell)}{\ell} = -\frac{4\nu}{\ell}\Gamma_{vor}'(\ell) - \frac{4}{\ell^2}\int_0^\ell ra_{vor}(r)\;dr
    \]
    \textbf{Step 1:} We begin with the forcing term and show that 
    \[
         \frac{-4}{\ell^2}\int_0^\ell ra_{vor}(r)\;dr = -2\eta + o_{\ell \to 0}(1)
    \]
    Note that 
    \[
        a_{vor}(0) = \frac{1}{2}\sum_{j}\E\fint_{\T_\lambda^2} |\curl f_j|^2\;dx = \eta.
    \]
    Furthermore, as the $\curl f^k$ are smooth we Taylor expand them around $r = 0$ and use the Peano formulation of the error:
    \begin{align*}
        \frac{-4}{\ell^2}\int_0^\ell ra_{vor}(r)\;dr &=  \frac{-4}{\ell^2}\int_0^\ell ra_{vor}(0)\;dr + \frac{-4}{\ell^2}\int_0^\ell r^2h_1(r)\;dr\\
        &= -2\varepsilon + \frac{-4}{\ell^2}\int_0^\ell r^2h_1(r)\;dr
    \end{align*}
    where $\lim_{r \to 0} h_1(r) = 0$. Note, this implies that $\frac{1}{\ell^2}\int_0^\ell r^2h_1(r)\;dr = o_{\ell \to 0}(1)$.
    \textbf{Step 2:} Next we show that 
    \[
        \frac{-4\nu}{\ell}\Gamma_{vor}'(\ell) = 2\eta - \nu\sum_{k}c(\ell,k)\E|\widehat{\grad \omega^\nu}|^2
    \]
    where $\ds\lim_{\ell|k| \to 0}c(\ell, k) = 0$ and $\ds\lim_{\ell|k| \to \infty}c(\ell,k) = 2$.. 

    From Proposition \ref{prop: KHM relatons and regularity} we know that $\Gamma_{vor} \in C^3$, we Taylor expand $\Gamma'$ about 0 using the Lagrange formulation of the error:
    \[
        \frac{-4\nu}{\ell}\Gamma_{vor}'(\ell) = \frac{-4\nu}{\ell}\Gamma_{vor}'(0) + \frac{-4\nu}{\ell}\int_0^\ell \Gamma_{vor}''(r)\;dr
    \]
    Due to the fact that $\fint_{S^1} n\;dS(n) = 0$ it follows from Fubini that 
    \[
        \Gamma_{vor}'(0) = \E\fint_0^T\fint_{S^1}\fint_{\T_\lambda^2} n_i\partial_i\omega^\nu \omega^\nu\;dxdS(n)\;dt = 0
    \]
    Consider the error term, which we will note is real valued. We apply Fourier transformation to the integral in $x$, Fubini, along with Lemma \ref{lemma: tangential term integral} when $p = 1$: 
    \begin{align*}
        \frac{-4\nu}{\ell}\int_0^\ell \Gamma_{vor}''(r)\;dr &= \frac{4\nu}{\ell}\int_0^\ell \E\fint_0^T\fint_{S^1}\fint_{\T_\lambda^2} n_in_j\partial_i\omega^\nu T_{rn}\partial_j\omega^\nu\;dxdS(n)dtdr\\
        &= \frac{-4\nu}{\ell}\sum_k\int_0^\ell \E\fint_0^T\fint_{S^1} (n\cdot k)^2e^{-irn\cdot k}\E|\widehat{\omega^\nu}|^2\;dS(n)dtdr\\
        &= -\frac{4\nu}{\ell}\sum_{k} \int_0^\ell |k|^{2}\sum_{m=0}^\infty \frac{(-1)^m (r|k|)^{2m}(2m+2)!}{4^{m+1} (2m)!(m+1)!(m+1)!}\;dr\fint_0^T\E|\widehat{\omega^\nu}|^2dt\\
        &= 2\nu\sum_{k}\sum_{m=0}^\infty \frac{(-1)^m (\ell|k|)^{2m}}{4^m(m+1)!m!}\fint_0^T\E|\widehat{\grad\omega^\nu}|^2dt\\
        &= 2\nu\sum_k \E\fint_0^T|\widehat{\grad \omega^\nu}|^2\;dt - 2\nu\sum_{k}\sum_{m=1}^\infty \frac{(-1)^{m+1} (\ell|k|)^{2m}}{4^m(m+1)!m!}\E\fint_0^T|\widehat{\grad\omega^\nu}|^2dt\\
        &= 2\eta - \nu\sum_{k}c(\ell,k)\E\fint_0^T|\widehat{\grad\omega^\nu}|^2dt
    \end{align*}
    Note that the application of Lemma \ref{lemma: tangential term integral} we use the definition of $\beta_2(m+1)$ and simplified the coefficient. Then in the ultimate equality we have used the enstrophy balance \eqref{eq: enstrophy balance} to get the $2\eta$ for the first component. It is clear from the definition of $c(\ell,k)$ that $\ds\lim_{\ell|k| \to 0} c(\ell,k) = 0$. Moreover, using the power series representation of $J_1(x)$\cite{abramowitz1988handbook} (note $J_1$ is the first order Bessel function of the first kind) one can write
    \[
        |2 - c(\ell,k)| = \Big|2\sum_{m=0}^\infty \frac{(-1)^m (\ell|k|)^{2m}}{4^m(m+1)!m!}\Big| = \Big|4\frac{J_1(\ell|k|)}{\ell|k|}\Big|.
    \]
    Since $J_1(x)/x \to 0$ as $x \to \infty$, this implies that $\ds\lim_{\ell|k| \to \infty} c(\ell, k) = 2$.
    
    \textbf{Step 3:} Next we apply Lemma \ref{lemma: small scale wave-number sum limit} to conclude that there exists $\ell_\nu \to 0$ such that 
    \begin{align*}
        \lim_{\ell_I \to 0}\limsup_{\nu \to 0}\sup_{\ell \in [\ell_\nu, \ell_I]} \frac{-4\nu}{\ell}\Gamma_{vor}'(\ell) &= 2\eta - \lim_{\ell_I \to 0}\liminf_{\nu \to 0}\inf_{\ell \in [\ell_\nu, \ell_I]}\nu\sum_{k}c(\ell,k)\E\fint_0^T|\widehat{\grad \omega^\nu}|^2\;dt\\
        &= 2\eta - 2\eta^*
    \end{align*}
    if and only if there exists $N_\nu \to \infty$ such that 
    \[
        \liminf_{\nu \to 0}\nu\sum_{|k| \geq N_\nu}\E\fint_0^T|\widehat{\grad \omega^\nu}|^2\;dt = \eta^*
    \]
    Finally, combine Steps 1-3 with the KHM equation to derive \eqref{eq: vorticity limit}.
    
\subsection{Proof of \eqref{eq: 2d S0 limit}}\label{sec: 2d direct S0}
    First, recall the KHM equation for $S_{vel}:$
    \[
        \frac{S_{vel}}{\ell^3} = -\frac{4\nu}{\ell^3}\Gamma_{vel}'(\ell) - \frac{4}{\ell^4}\int_0^\ell ra_{vel}(r)\;dr
    \]
    \textbf{Step 1:} Once again, we begin with the forcing term and show that 
    \[
        \frac{-4}{\ell^4}\int_0^\ell ra_{vel}(r)\;dr = \frac{-2\varepsilon}{\ell^2} + \frac{\eta}{4} + o_{\ell \to 0}(1)
    \]
    Recall that $a_{vel}$ is defined by
    \[
        a_{vel}(r) = \frac{1}{2}\sum_j\fint_S\fint_{\T_\lambda^2}f_j \cdot T_{rn}f_j\;dxdS(n)
    \]
    As the $f_j$ are smooth we Taylor expand them around $r=0$ and use the Peano formulation of the error (this is because we don't require an in-depth analysis of the error term for the forcing term).
    Moreover, as $\fint_S n_i\;dS(n) = 0$ it follows from Fubini that 
    \[
        \frac{1}{2}\sum_j\fint_S\fint_{\T_\lambda^2} f_j \cdot n_i\partial_if_j\;dxdS(n) = 0.
    \]
    Then by integration by parts we obtain
    \begin{align*}
        \fint_{\T_\lambda^2}f_j\cdot T_{rn}f_j\;dx &= \fint_{\T_\lambda^2}|f_j|^2 + \frac{r^2}{2}n_in_mf_j\cdot \partial_i\partial_m f_j + r^2 h_2(r)\;dx\\
        &= \fint_{\T_\lambda^2}|f_j|^2 - \frac{r^2}{2}n_in_m \partial_if_j\cdot \partial_m f_j\;dx + r^2 h_2(r)
    \end{align*}
    where $\ds\lim_{r \to 0}h_2(r) = 0$.
    Since $\fint_S n_in_m = \frac{1}{2}\delta_{im}$ we use the definition of $\eta$ and $\varepsilon$ to conclude that 
    \begin{align*}
        \frac{-4}{\ell^4}\int_0^\ell ra_{vel}(r)\;dr &= \frac{-2}{\ell^4}\int_0^\ell r\sum_j\fint_{\T_\lambda^2}|f_j|^2\;dxdr + \frac{1}{2\ell^4}\int_0^\ell r^3\sum_j\fint_{\T_\lambda^2}|\grad f_j|^2\;dxdr + \frac{-4}{\ell^4}\int_0^\ell r^3 h_2(r)\;dx\\
        &= \frac{-4}{\ell^4}\int_0^\ell \varepsilon r\;dr + \frac{1}{\ell^4}\int_0^\ell r^3\eta\;dr + \frac{-4}{\ell^4}\int_0^\ell r^3 h_2(r)\;dr\\
        &= \frac{-2\varepsilon}{\ell^2} + \frac{\eta}{4} + \frac{-4}{\ell^4}\int_0^\ell r^3 h_2(r)\;dr
    \end{align*}
    Finally, note that since $\ds\lim_{r \to 0}h_2(r) = 0$ the term $\frac{1}{\ell^4}\int_0^\ell r^3h_2(r)\;dr = o_{\ell \to 0}(1)$.
    
    \textbf{Step 2:} Next we show that 
    \[
        \frac{-4\nu}{\ell^3}\Gamma_{vel}'(\ell) = \frac{2\varepsilon}{\ell^2} - \frac{\eta}{4} - \nu\sum_k c(\ell, k)\E|\widehat{\grad \omega^\nu}|^2
    \]
    where $\ds\lim_{\ell|k| \to 0}c(\ell,k) = 0$ and $\ds\lim_{\ell|k| \to \infty}c(\ell, k) = \frac{-1}{4}$ as $\ell|k| \to \infty$.
    
    From Proposition \ref{prop: KHM relatons and regularity} we know that $\Gamma_{vel} \in C^4$, thus we Taylor expand about 0 while using the Lagrange formulation of the error:
    \[
        \frac{-4\nu}{\ell^3}\Gamma_{vel}'(\ell) = \frac{-4\nu}{\ell^3}\Gamma_{vel}'(0) + \frac{-4\nu}{\ell^2}\Gamma_{vel}''(0) + \frac{-2\nu}{\ell}\Gamma_{vel}'''(0) +  \frac{-2\nu}{\ell^3}\int_0^\ell (\ell - s)^2\Gamma_{vel}''''(s)\;ds
    \]
    From Lemma \ref{lemma: isotropic tensor identity} we have $\fint_{S^1} n\;dS(n) = 0$, $\fint_{S^1} n_in_j\;dS(n) = \frac{1}{2}\delta_{ij}$, and $\fint_{S^1} n_in_jn_k\;dS(n) = 0$, hence by Fubini 
    \begin{align*}
         \Gamma_{vel}'(0) &= -\E\fint_S\fint_{\T_\lambda^2} n_i\partial_iu_m^\nu u_m^\nu\;dxdS(n) = 0\\
         \Gamma_{vel}''(0) &= -\E\fint_S\fint_{\T_\lambda^2} n_in_j \partial_iu_m^\nu \partial_ju_m^\nu\;dxdS(n) = -\frac{1}{2}\E\|\grad u^\nu\|_\lambda^2\\
         \Gamma_{vel}'''(0) &= \E\fint_S\fint_{\T_\lambda^2} n_in_jn_k \partial_i\partial_ju_m^\nu \partial_ku_m^\nu\;dxdS(n) = 0
    \end{align*}
    
    Considering the error term, which we note is real valued. We use Fourier analysis to convert the integral in $x$ to a Fourier series in $k$ (wave-number), and apply Fubini along with Lemma \ref{lemma: tangential term integral} when $p = 2$:   
    \begin{align*}
        \int_0^\ell \frac{-2\nu(\ell -r)^2}{\ell^3}\Gamma_{vel}''''(r)\;dr &= \int_0^\ell \frac{-2\nu(\ell -r)^2}{\ell^3}\E\fint_S\fint_{\T_\lambda^2} n_in_jn_mn_p\partial_m\partial_iu_q^\nu\partial_p\partial_jT_{rn}u_q^\nu\;dxdS(n)dr\\
        &= \sum_k\int_0^\ell \frac{-2\nu(\ell -r)^2}{\ell^3} \fint_S (n \cdot k)^4e^{-irn\cdot k}\E|\widehat{u^\nu}|^2\;dS(n)dr\\
        &= \sum_k \int_0^\ell \frac{-2\nu(\ell -r)^2}{\ell^3}\frac{2|k|^{4}}{4^2}\sum_{m=0}^\infty \frac{(-1)^m (2m+3)!(r|k|)^{2m}}{4^m (2m)!(m+2)!(m+1)!}\E|\widehat{u^\nu}(k)|^2\\  
        &= \frac{-\nu}{4}\sum_k\sum_{m=0}^\infty \frac{(-1)^m (2m+3)!|k|^{2m}}{4^m (2m)!(m+2)!(m+1)!}\int_0^\ell \frac{(\ell - r)^2r^{2m}}{\ell^3}\;dr\E|\widehat{\grad \omega^\nu}(k)|^2\\
        &= \frac{-\nu}{2}\sum_k\sum_{m=0}^\infty \frac{(-1)^m (\ell|k|)^{2m}}{4^m(m+2)!(m+1)!}\E|\widehat{\grad \omega^\nu}(k)|^2\\
        &= \frac{-\nu}{4}\sum_k \E|\widehat{\grad \omega^\nu}(k)|^2 - \frac{\nu}{2}\sum_k\sum_{m=1}^\infty \frac{(-1)^m (\ell|k|)^{2m}}{4^m(m+2)!(m+1)!}\E|\widehat{\grad \omega^\nu}(k)|^2\\
        &=: \frac{-\eta}{4} - \nu\sum_k c(\ell,k)\E|\widehat{\grad \omega^\nu}(k)|^2
    \end{align*}
    Note that when applying Lemma \ref{lemma: tangential term integral} we use the definition of $\beta_2(m+2)$ and simplified the coefficient. Furthermore, we can relate $\widehat{u^\nu}$ and $\widehat{\omega^\nu}$ via the Biot-Savart law:
    \[
        |k|^4|\widehat{u^\nu}|^2 = |k|^4\frac{-|k|^2}{|k|^4}|\widehat{\omega^\nu}|^2 = -|k|^2|\widehat{\omega^\nu}|^2 = |\widehat{\grad \omega^\nu}|^2
    \]
    Then in the ultimate equality we have used the enstrophy balance \eqref{eq: enstrophy balance} to get the $\eta/4$ for the first component. It is clear from the definition of $c(\ell,k)$ that $\ds\lim_{\ell|k| \to 0} c(\ell,k) = 0$. 
    Moreover,
    \begin{align*}
        |\frac{-1}{4} - c(\ell,k)| &= \Big|\frac{-1}{2}\sum_{m=0}^\infty \frac{(-1)^m (\ell|k|)^{2m}}{4^m(m+2)!(m+1)!}\Big| = \Big|\frac{1}{(\ell|k|)^2}\sum_{m=0}\frac{(-1)^m(\ell|k|)^{2m + 2}}{4^m(m+2)!m!(2m+2)}\Big|\\
        &= \Big|\frac{1}{(\ell|k|)^2}\int_0^{\ell|k|}\sum_{m=0}\frac{(-1)^m t^{2m + 3}}{4^m(m+2)!m!}\;dt\Big| = \Big|\frac{4}{(\ell|k|)^2}\int_0^{\ell|k|}tJ_2(t)\;dt\Big|\\
        &\leq \frac{4}{(\ell|k|)^2}\int_0^{\ell|k|}t|J_2(t)|\;dt \leq \frac{C}{(\ell|k|)^2}\int_0^{\ell|k|}t^{1/2}\;dt\\
        &= \frac{C}{(\ell|k|)^{-1/2}}
    \end{align*}
    where we have used the power series representation of the Bessel function $J_2(x)$ and that there exists $C> 0$ such that $|J_2(x)|\leq C|x|^{-1/2}$ for all $x$ \cite{abramowitz1988handbook}. Hence $\ds\lim_{\ell|k| \to \infty} c(\ell,k) = \frac{-1}{4}$.
        
    \textbf{Step 3:} Next we apply Lemma \ref{lemma: small scale wave-number sum limit} to conclude that there exists $\ell_\nu \to 0$ such that 
    \begin{align*}
        \lim_{\ell_I \to 0}\limsup_{\nu \to 0}\sup_{\ell \in [\ell_\nu, \ell_I]} \int_0^\ell \frac{-2\nu(\ell -s)^2}{\ell^3}\Gamma''''(s)\;ds &= \frac{-\eta}{4} - \lim_{\ell_I \to 0}\liminf_{\nu \to 0}\inf_{\ell \in [\ell_\nu, \ell_I]}\nu\sum_{k}c(\ell,k)\E|\widehat{\grad \omega}|^2\\
        &= \frac{-\eta}{4} + \frac{\eta^*}{4}
    \end{align*}
    if and only if there exists $N_\nu \to \infty$ such that 
    \[
        \liminf_{\nu \to 0}\nu\sum_{|k| \geq N_\nu}\E|\widehat{\grad \omega}|^2 = \eta^*
    \]
    Finally, combine Steps 1-3 with the KHM equation to derive \eqref{eq: 2d S0 limit}.
    
\subsection{Proof of \eqref{eq: 2d S long limit}}\label{sec: 2d direct S long}
 Recall the KHM equation for $S_{vel}^\parallel:$
    \[
        \frac{S_{vel}^\parallel(\ell)}{\ell^3} = -\frac{4\nu}{\ell^3}(\Gamma_{vel}^\parallel)'(\ell) + \frac{2}{\ell^6}\int_0^\ell r^2S_{vel}(r)\;dr - \frac{4}{\ell^6}\int_0^\ell r^3a_{vel}^\parallel(r)\;dr
    \]
    \textbf{Step 1:} Once again, we begin with the energy injection term and show that 
    \[
        \frac{-4}{\ell^6}\int_0^\ell r^3a_{vel}^\parallel(r)\;dr = \frac{-\varepsilon}{2\ell^2} + \frac{\eta}{24} + o_{\ell \to 0}(1)
    \]
    This follows exactly the same as in \cite{bedrossian2019sufficient} but we will include it for the sake of completeness.
    Recall that 
    \[
        a_{vel}^\parallel(r) = \frac{1}{2}\sum_j\fint_{S^1}\fint_{\T_\lambda^2}(n\cdot f_j)(n \cdot T_{rn}f_j)\;dxdS(n) 
    \]
    As the $f_j$ are smooth, we Taylor expand about $r=0$ and use the Peano formulation of the error. After integration by parts we have
    \begin{align*}
        \fint_{\T_\lambda^2} (n\cdot f_j)(n \cdot T_{rn}f_j)\;dx &= \fint_{\T_\lambda^2}(n\cdot f_j)^2 - \frac{r^2}{2}n_in_kn_pn_q\partial_k(g_j)_{i}\partial_p(g_j)_{q}\;dx + r^3h_2(r)
    \end{align*}
    where $\ds\lim_{r \to 0}h_2(r) = 0$. Note that this implies $\frac{1}{\ell^6}\int_0^\ell r^6h_2(r)\;dr = o_{\ell \to 0}(1)$.
    From Lemma \ref{lemma: isotropic tensor identity} we have the identities $\fint_S n_in_j\;dS(n) = \frac{1}{2}\delta_{ij}$ and 
    \[
        \fint_S n_in_kn_pn_q\;dS(n) = \frac{1}{8}(\delta_{ik}\delta_{pq} + \delta_{ip}\delta_{kq} + \delta_{iq}\delta_{pk})
    \]
    Since the vector fields $f_j$ are divergence free it follows that
    \begin{align*}
        \sum_j \frac{2}{\ell^6}\int_0^\ell r^3\fint_S\fint_{\T_\lambda^2}(n\cdot f_j)^2\;dxdS(n)dr &= \frac{1}{\ell^6}\int_0^\ell r^3 \varepsilon = \frac{\varepsilon}{2\ell^2}\\
        -\sum_j \frac{2}{\ell^6}\int_0^\ell r^5\fint_S\fint_{\T_\lambda^2}n_in_kn_pn_q\partial_k(f_j)_{i}\partial_p(f_j)_{q}\;dxdS(n)dr &= \frac{-1}{4\ell^6}\sum_j\int_0^\ell r^5\fint_{\T_\lambda^2}|\grad g_j|^2\;dxdr = \frac{-\eta}{24}
    \end{align*}
    Then putting everything together we get
    \[
        \frac{-4}{\ell^6}\int_0^\ell r^3 a_{vel}^\parallel(r)\;dr = \frac{-2}{\ell^6}\int_0^\ell r^3 \sum_j \fint_S\fint_{\T_\lambda^2} (n\cdot f_j)(n \cdot T_{rn}f_j)\;dxdS(n)dr = -\frac{\varepsilon}{2\ell^2} + \frac{\eta}{24} + o_{\ell \to 0}(1)
    \]
    \textbf{Step 2:} Next we show that 
    \[
        \frac{-4\nu}{\ell^3}(\Gamma_{vel}^\parallel)'(\ell) = \frac{\varepsilon}{2\ell^2} - \frac{\eta}{24} + \sum_{k}c(\ell,k)\E|\widehat{\grad \omega}(k)|^2
    \]
    where the coefficients satisfy $\ds\lim_{\ell|k| \to 0}c(\ell,k) = 0$ and $\ds\lim_{\ell|k| \to \infty}c(\ell,k) = \frac{1}{24}$

    From Proposition \ref{prop: KHM relatons and regularity}, $\Gamma_\parallel \in C^4$ so we Taylor expand it about $r=0$ with the Lagrange formulation of the error to get
    \[
        \frac{-4\nu}{\ell^3}(\Gamma_{vel}^\parallel)'(\ell) =  \frac{-4\nu}{\ell^3}(\Gamma_{vel}^\parallel)'(0) - \frac{4\nu}{\ell^2}(\Gamma_{vel}^\parallel)''(0) - \frac{2\nu}{\ell}(\Gamma_{vel}^\parallel)'''(0) - \frac{2\nu}{\ell^3}\int_0^\ell (\ell - r)^2 (\Gamma_{vel}^\parallel)''''(r)\;dr
    \]
    We know from Lemma \ref{lemma: isotropic tensor identity} that $\fint_{S^1} n_in_jn_k\;dS(n) = \fint_{S^1} n_in_jn_kn_pn_q \;dS(n) = 0$ and $\fint_{S^1} n_in_jn_kn_m\;dS(n) = \frac{1}{8}(\delta_{ij}\delta_{km} + \delta_{ik}\delta_{jm} + \delta_{im}\delta_{jk})$. Now we apply Fubini's theorem and the energy balance \eqref{eq: energy balance} to see 
    \begin{align*}
        \frac{-4\nu}{\ell^3}(\Gamma_{vel}^\parallel)'(0) &= \frac{4\nu}{\ell^3}\E\fint_{S^1}\fint_{\T_\lambda^2}n_in_jn_k\partial_ku_i^\nu u_j^\nu\;dxdS(n) = 0\\
        \frac{-4\nu}{\ell^2}(\Gamma_{vel}^\parallel)''(0) &= \frac{4\nu}{\ell^2}\E\fint_{S^1}\fint_{\T_\lambda^2}n_in_jn_kn_p\partial_ku_i^\nu\partial_pu_j^\nu\;dxdS(n) = \frac{\nu}{2\ell^2}\E\|\grad u^\nu\|_{\lambda}^2 = \frac{\varepsilon}{2\ell^2}\\
        \frac{-2\nu}{\ell}(\Gamma_{vel}^\parallel)'''(0) &= \frac{2\nu}{\ell}\E\fint_{S^1}\fint_{\T_\lambda^2}n_in_jn_kn_pn_q\partial_k\partial_pu_i^\nu\partial_qu_j^\nu\;dxdS(n) = 0
    \end{align*}
    Next, we consider the error term, which we will note is real valued. Using Fourier Analysis, Fubini's theorem, and Lemma \ref{lemma: longitudinal term integral} when $p = 2$ it follows that
    \begin{align*}
        \frac{2\nu}{\ell^3}\int_0^\ell (\ell - r)^2(\Gamma_{vel}^\parallel)''''(r)\;dr &= \sum_k\frac{2\nu}{\ell^3}\int_0^\ell (\ell - r)^2\fint_S n_in_j(n \cdot k)^4e^{-ir(n\cdot k)}\E\widehat{u_i}(k)\widehat{u_j}(k)\;dS(n)dr\\
        &= \sum_k\frac{2\nu}{\ell^3}\int_0^\ell (\ell - r)^2\frac{2|k|^2}{4^2}\sum_{m=0}^\infty \frac{(-1)^m(2m+3)!(r|k|)^{2m}}{4^m(m+1)!(m+3)!(2m)!}\E|\widehat{\omega}(k)|^2\\
        &= \frac{\nu}{4}\sum_k\sum_{m=0}^\infty \frac{(-1)^m(\ell|k|)^{2m}}{4^m(m+1)!(m+3)!}\E|\widehat{\grad\omega}(k)|^2\\
        &= \frac{\nu}{24}\E|\widehat{\grad \omega}(k)|^2 - \frac{1}{4}\nu\sum_k\sum_{m=1}^\infty \frac{(-1)^{m+1}(\ell|k|)^{2m}}{4^m(m+1)!(m+3)!}\E|\widehat{\grad\omega}(k)|^2\\
         &=: \frac{\eta}{24} - \sum_kc(\ell,k)\E|\widehat{\grad \omega}(k)|^2
    \end{align*}
    Note that we applied the definition of $\beta_2(m+3)$ and simplified the coefficient when applying Lemma \ref{lemma: longitudinal term integral}. Also by construction $\ds\lim_{\ell|k| \to 0}c(\ell,k) = 0$. Moreover, 
    \begin{align*}
        \Big|\frac{1}{24} - c(\ell,k)\Big| &= \Big|\frac{1}{4}\sum_{m=0}^\infty \frac{(-1)^m(\ell|k|)^{2m}}{4^m(m+1)!(m+3)!}\Big| = \Big|\frac{1}{2(\ell|k|)^2}\sum_{m=0}\frac{(-1)^m(\ell|k|)^{2m + 2}}{4^mm!(m+3)!(2m+2)}\Big|\\
        &= \Big|\frac{1}{2(\ell|k|)^2}\int_0^{\ell|k|}\sum_{m=0}\frac{(-1)^m t^{2m + 3}}{4^mm!(m+3)!}\;dt\Big| = \Big|\frac{4}{(\ell|k|)^2}\int_0^{\ell|k|}J_3(t)\;dt\Big|\\
        &\leq \frac{4}{(\ell|k|)^2}\int_0^{\ell|k|}|J_3(t)|\;dt \leq \frac{C}{\ell|k|}
    \end{align*}
    where we have used the power series representation of the Bessel function $J_3(x)$ and the fact that there exists $C > 0$ such that $|J_3(x)| \leq C$ uniformly  for all $x$\cite{abramowitz1988handbook}. Hence $\ds\lim_{\ell|k| \to \infty} c(\ell,k) = \frac{1}{24}$.

    \textbf{Step 3:} Apply Lemma \ref{lemma: small scale wave-number sum limit} to conclude that there exists $\ell_\nu \to 0$ such that 
    \begin{align*}
         \frac{-4\nu}{\ell^3}(\Gamma_{vel}^\parallel)'(\ell) = \frac{\varepsilon}{2\ell^2} - \frac{\eta}{24} - \frac{\eta^*}{24} + o_{\ell \to 0}(1)
    \end{align*}
    if and only if there exists $N_\nu \to \infty$ such that 
    \[
        \liminf_{\nu \to 0}\nu\sum_{|k| \geq N_\nu}\E|\widehat{\grad \omega^\nu}|^2 = \eta^*.
    \]
    \textbf{Step 4:} Next we use Equation \eqref{eq: 2d S0 limit} which was proved in Section \ref{sec: 2d direct S0}, to conclude that the if and only if established in step 3, also equivalently shows that
    \[
        \frac{2}{\ell^6}\int_0^\ell r^2S_{vel}(r)\;dr = \frac{2}{\ell^6}\int_0^\ell r^5\frac{S_{vel}(r)}{r^3}\;dr = \frac{\eta^*}{12} + o_{\ell \to 0}(1)  
    \]

    Finally combine Steps 1-4 with the KHM equation to derive \eqref{eq: 2d S long limit}.

\section{The Inverse Cascade}\label{sec: 2d inverse cascade}
In this section, we characterize the existence of an inverse cascade in both 2D and 3D.
\begin{theorem}[Characterization of the Inverse Cascade]
    Suppose that $\lambda = \lambda(\nu) < \infty$ is a continuous monotone increasing function such that $\ds\lim_{\nu \to 0}\lambda = \infty$. Let $\{u\}_{\nu > 0}$ be a sequence of statistically stationary solutions to the Navier-Stokes equations \eqref{eq: incompressible NSE} on $\T_\lambda^d$ with mean-zero, divergent free forcing correlations $f_j$. Then there exists a decreasing sequence $M_\nu$ such that $\ds\lim_{\nu \to 0}M_\nu = 0$ and 
    \[
       \liminf_{\nu \to 0}\nu\sum_{|k| \leq M_\nu}\E|\widehat{\grad u^\nu}(k)|^2 = \varepsilon^*
    \]
    if and only if there exists $\tilde{\ell}_\nu \in (1, \lambda)$ such that $\ds \lim_{\nu \to 0}\tilde{\ell}_\nu = \infty$ and
    \begin{align}
        \lim_{\ell_I \to \infty}\limsup_{\nu \to 0}\sup_{\ell \in [\ell_I, \tilde{\ell}_\nu]}\Big|\frac{S_{vel}(\ell)}{\ell} - 4\beta_d(1)\varepsilon^*\Big| = 0\label{eq: S0 limit inverse}\\
        \lim_{\ell_I \to \infty}\limsup_{\nu \to 0}\sup_{\ell \in [\ell_I, \tilde{\ell}_\nu]}\Big|\frac{S_{vel}^\parallel(\ell)}{\ell} - \big(4\beta_d(2) + \frac{8\beta_d(1)}{d+2}\big)\varepsilon^*\Big| = 0\label{eq: S long limit inverse}
    \end{align}
\end{theorem}
Explicitly the coefficients can be computed as 
\[
    4\beta_d(1) = \begin{cases}
        2 & d = 2\\
        4/3 & d = 3
    \end{cases}
    \quad 4\beta_d(2) + \frac{8\beta_d(1)}{d+2} = \begin{cases}
        3/2 & d = 2\\
        4/5 & d = 3
    \end{cases}
\]
The proof of Theorem \ref{thm: inverse cascade characterization} is split into two different subsections, one for each of the flux laws. While the details are slightly different for each one, similar to the direct cascade case the mechanics are similar for both flux laws.  
\subsection{Proof of \eqref{eq: S0 limit inverse}}\label{sec: proof of S0 inverse}
    First, recall the KHM equation for $S_{vel}$:
    \[
         \frac{S_{vel}(\ell)}{\ell} = -\frac{4\nu}{\ell}\Gamma_{vel}'(\ell) - \frac{4}{\ell^d}\int_0^\ell r^{d-1}a_{vel}(r)\;dr
    \]
    \textbf{Step 1:} Beginning with the source term $a_{vel}(r)$, we apply Fourier analysis, Fubini and Lemma \ref{lemma: tangential term integral} with $p = 0$ to derive:
    \begin{align*}
        \frac{4}{\ell^d}\int_0^\ell r^{d-1}a_{vel}(r)\;dr &= \frac{2}{\ell^d}\sum_j\sum_k\int_0^\ell \fint_{\S^{d-1}}r^{d-1}e^{-irn\cdot k}\;dS(n)dr|\widehat{f_j}|^2\\
        &= \frac{2}{\ell^d}\sum_j\sum_k\int_0^\ell r^{d-1}\sum_{m=0}^\infty \beta_d(m)\frac{(-1)^m(r|k|)^{2m}}{2^mm!}\;dr|\widehat{f_j}|^2\\
        &= 2\sum_j\sum_k \sum_{m=0}^\infty \beta_d(m)\frac{(-1)^m(\ell|k|)^{2m}}{2^m(2m+d)m!}|\widehat{f_j}|^2\\
        &:= \frac{2\beta_d(0)}{d}\sum_j\|f_j\|_{L^2}^2 + \sum_j\sum_k c(\ell,k)|\widehat{f_j}|^2
    \end{align*}
    Recall that when $d=2$, $\beta_2(m) = \ds\frac{1}{2^mm!}$ so
    \[
        \big|\frac{2\beta_2(0)}{2} + c(\ell,k)\big| = \Big|\sum_{m=0}^\infty \frac{(-1)^m(\ell|k|)^{2m}}{4^m(m+1)!m!}\Big| = \Big|\frac{J_1(\ell|k|)}{\ell|k|}\Big| \to 0
    \]
    as $\ell|k| \to \infty$. Here we have used the power series representation of the Bessel function of the first kind and that $|J_1(x)|<1$ uniformly for all $x\in \R$\cite{abramowitz1988handbook}. Alternatively when $d=3$, recall that $\beta_3(m) = \frac{2^mm!}{(2m+1)!}$ so
    \begin{align*}
        \big|\frac{2\beta_d(0)}{d} + c(\ell,k)\big| &= \Big|\sum_{m=0}^\infty \frac{(-1)^m(\ell|k|)^{2m}}{(2m+3)(2m+1)!}\Big|\\
        &= \Big|\frac{2}{(\ell|k|)^3}\int_0^{\ell|k|} \sum_{m=0}^\infty \frac{(-1)^mr^{2m+2}}{(2m+1)!}\;dr\Big|\\
        &= \Big|\frac{2}{(\ell|k|)^3}\int_0^{\ell|k|} r\sin(r)\;dr\Big|\\
        &\leq \frac{1}{\ell|k|} \to 0
    \end{align*}
    as $\ell|k| \to \infty$.
    Hence in both cases $c(\ell,k) \to -\frac{2\beta_d(0)}{d}$. 
    By Lemma \ref{lemma: large scale wave-number sum limit} it follows that 
    \begin{align*}
        \lim_{\ell_I \to \infty}\limsup_{\nu \to 0}\sup_{\ell \in (\ell_I, \tilde{\ell_\nu})} \frac{4}{\ell^d}\int_0^\ell r^{d-1}a_{vel}(r)\;dr &=
        \frac{2\beta_d(0)}{d}\sum_j\|f_j\|_{L^2}^2 -\frac{2\beta_d(0)}{d}\sum_j(\|f_j\|_{L^2}^2 - |\widehat{f_j}(0)|^2)\\
        &= \frac{2\beta_d(0)}{d}\sum_j|{f_j}(0)|^2 = 0
    \end{align*}
    as the $f_j$ are mean zero.
    
    \textbf{Step 2:} Next, we use Fourier analysis to translate the energy correlation term $\Gamma_{vel}$ as
    \[
        \frac{-4\nu}{\ell}\Gamma_{vel}'(\ell) = 4\beta_d(1)\varepsilon + \nu\sum_{k}c(\ell,k)\E\fint_0^T|\widehat{\grad u^\nu}(k)|^2
    \]
    where $\ds\lim_{\ell|k| \to 0}c(\ell,k) = 0$ and $\ds\lim_{\ell|k| \to \infty}c(\ell,k) = -4\beta_d(1)$. 
    
    By Proposition \ref{prop: KHM relatons and regularity} we know $\Gamma_{vel} \in C^4$. Thus we Taylor expand $\Gamma_{vel}'$ about 0 and use the Lagrange formulation of the error:
    \[
        \frac{-4\nu}{\ell}\Gamma_{vel}'(\ell) = \frac{-4\nu}{\ell}\Gamma_{vel}'(0) + \frac{-4\nu}{\ell}\int_0^\ell \Gamma_{vel}''(r)\;dr
    \]
    From the proof of Lemma \ref{lemma: isotropic tensor identity} we know that $\fint_{S^{d-1}} n\;dS(n) = 0$, so after applying Fubini's theorem we have
    \[
        \Gamma_{vel}'(0) = -\E\fint_{S^{d-1}}\fint_0^T\fint_{\T_\lambda^d} n_i\partial_iu_m^\nu u_m^\nu\;dxdtdS(n) = 0
    \]
    Next, we consider the error term, which is real valued. We use Plancheral's theorem to rewrite the integral in $x$ in terms of the Fourier series and apply both Fubini and Lemma \ref{lemma: tangential term integral} when $p = 1$ to get:
    \begin{align*}
        \frac{-4\nu}{\ell}\int_0^\ell \Gamma_{vel}''(r)\;dr &= \frac{-4\nu}{\ell}\sum_k\int_0^\ell \E\fint_{S^{d-1}} (n\cdot k)^2e^{-irn\cdot k}\E\fint_0^T|\widehat{u^\nu}|^2\;dtdS(n)dr\\
        &= -\frac{4\nu}{\ell}\sum_{k} \int_0^\ell |k|^{2}\sum_{m=0}^\infty \beta_d(m+1) \frac{(-1)^m (r|k|)^{2m}(2m+2)!}{2^{m+1} (2m)!(m+1)!}\;dr\E\fint_0^T|\widehat{u^\nu}|^2\\
        &= 4\nu\sum_{k} \sum_{m=0}^\infty \beta_d(m+1)\frac{(-1)^m (\ell|k|)^{2m}}{2^{m}m!}\E\fint_0^T|\widehat{\grad u^\nu}|^2\\
        &= 4\beta_d(1)\nu\sum_{k}\E\fint_0^T|\widehat{\grad u^\nu}|^2 + 4\nu\sum_{k} \sum_{m=1}^\infty \beta_d(m+1)\frac{(-1)^m (\ell|k|)^{2m}}{2^{m}m!}\E\fint_0^T|\widehat{\grad u^\nu}|^2\\
        &=: 4\beta_d(1)\varepsilon - 4\beta_d(1)\E D(u^\nu) + \sum_k c(\ell,k)\E|\widehat{\grad u^\nu}|^2
    \end{align*}
    Note that the final equality follows from the energy balance \ref{eq: energy balance}.
    
    It is clear from the definition of $c(\ell,k)$ that $\ds\lim_{\ell|k| \to 0}c(\ell,k) = 0$. Moreover, using the same argument as in step 1 we can show
    \[
        |4\beta_d(1) + c(\ell,k)| = \Big|4\sum_{m=0}^\infty \beta_d(m+1)\frac{(-1)^m (\ell|k|)^{2m}}{2^mm!}\Big| \leq \frac{C}{\ell|k|}
    \]
    Hence 
    \[
        \lim_{\ell|k| \to \infty}c(\ell,k) = -4\beta_d(1)
    \]
    
    \textbf{Step 3:} Apply Lemma \ref{lemma: large scale wave-number sum limit} to conclude that there exists $\tilde{\ell}_\nu \to \infty$ such that 
    \begin{align*}
        \lim_{\ell_I \to \infty}\limsup_{\nu \to 0}\sup_{\ell \in [\ell_I, \tilde{\ell}_\nu]} \frac{-4\nu}{\ell}\Gamma'(\ell) &= 2\varepsilon + \lim_{\ell_I \to \infty}\limsup_{\nu \to 0}\sup_{\ell \in [\ell_I, \tilde{\ell}_\nu]}\nu\sum_{k}c(\ell,k)\E|\widehat{\grad u}|^2\\
        &= 4\beta_d(1)(\varepsilon - \E D(u^\nu)) - 4\beta_d(1)\Big(\varepsilon - \E D(u^\nu) - \varepsilon^*\Big)\\
        &= 4\beta_d(1)\varepsilon^*
    \end{align*}
    if and only if there exists $M_\nu \to 0$ such that
    \[
        \liminf_{\nu \to 0}\nu\sum_{|k| \leq M_\nu}\E|\widehat{\grad u^\nu}|^2 = \varepsilon^*
    \]
    
    Finally, combine Steps 1-3 with the KHM equation to derive \eqref{eq: S0 limit inverse}.

\subsection{Proof of \eqref{eq: S long limit inverse}}
 Recall the KHM equation for $S_{vel}^\parallel:$
    \[
        \frac{S_{vel}^\parallel(\ell)}{\ell} = -\frac{4\nu}{\ell}(\Gamma_{vel}^\parallel)'(\ell) + \frac{2}{\ell^{d+2}}\int_0^\ell r^{d}S_{vel}(r)\;dr - \frac{4}{\ell^{d+2}}\int_0^\ell r^{d+1}a_{vel}^\parallel(r)\;dr
    \]
    \textbf{Step 1:} 
    Beginning with the source term $a_{vel}^\parallel(r)$, we apply Fourier analysis, Fubini and Lemma \ref{lemma: longitudinal term integral} with $p = 0$ to derive:
    \begin{align*}
        \frac{4}{\ell^{d+2}}\int_0^\ell r^{d+1}a_{vel}^\parallel(r)\;dr &= \frac{2}{\ell^{d+2}}\sum_j\sum_k\int_0^\ell \fint_{\S^{d-1}}r^{d+1}n_an_be^{-irn\cdot k}\;dS(n)dr\widehat{f_j}_a\widehat{f_j}_b\\
        &= \frac{2}{\ell^{d+2}}\sum_j\sum_k\sum_{m=0}^\infty \int_0^\ell \beta_d(m+1)r^{d+1}\frac{(-1)^m(r|k|)^{2m}}{2^mm!}\;dr|\widehat{f_j}|^2\\
        &= 2\sum_j\sum_k\sum_{m=0}^\infty \beta_d(m+1)\frac{(-1)^m(\ell|k|)^{2m}}{2^mm!(2m+d+2)}|\widehat{f_j}|^2\\
        &= \frac{2\beta_d(1)}{d+2}\sum_j\|f_j\|_{L^2}^2 + 2\sum_j\sum_k\sum_{m=0}^\infty \beta_d(m+1)\frac{(-1)^m(\ell|k|)^{2m}}{2^mm!(2m+d+2)}|\widehat{f_j}|^2\\
        &:= \frac{2\beta_d(1)}{d+2}\sum_j\|f_j\|_{L^2}^2 + \sum_j\sum_k c(\ell,k)|\widehat{f_j}|^2
    \end{align*}
    Recall that when $d=2$, $\beta_2(m) = \ds\frac{1}{2^mm!}$ so
    \[
        \big|\frac{2\beta_2(1)}{4} + c(\ell,k)\big| = \Big|\frac{1}{2}\sum_{m=0}^\infty \frac{(-1)^m(\ell|k|)^{2m}}{4^mm!(m+2)!}\Big| = \Big|\frac{2J_2(\ell|k|)}{(\ell|k|)^2}\Big| \to 0 \quad \text{ as } \ell|k| \to \infty
    \]
    Here we have used the power series representation of the Bessel function of the first kind and that $|J_2(x)|<1$ uniformly for all $x\in \R$\cite{abramowitz1988handbook}. Alternatively when $d=3$, recall that $\beta_3(m) = \frac{2^mm!}{(2m+1)!}$ so
    \begin{align*}
        \big|\frac{2\beta_d(1)}{5} + c(\ell,k)\big| &= \Big|2\sum_{m=0}^\infty \frac{(-1)^m(\ell|k|)^{2m}}{(2m+5)(2m+3)(2m+1)!}\Big|\\
        &= \Big|\frac{2}{(\ell|k|)^5}\int_0^{\ell|k|} \sum_{m=0}^\infty \frac{(-1)^mr^{2m+4}}{(2m+3)(2m+1)!}\;dr\Big|\\
        &= \Big|\frac{2}{(\ell|k|)^5}\int_0^{\ell|k|} r\int_0^r \sum_{m=0}^\infty \frac{(-1)^mq^{2m+2}}{(2m+1)!}\;dqdr\Big|\\
        &= \Big|\frac{2}{(\ell|k|)^5}\int_0^{\ell|k|} r\int_0^r q\sin(q)\;dqdr\Big|\\
        &\leq \frac{C}{(\ell|k|)} \to 0 \quad \text{ as } \ell|k| \to \infty
    \end{align*}
    Hence in both cases $\lim_{\ell|k| \to \infty}c(\ell,k) = -\frac{2\beta_d(1)}{d+2}$. Moreover, by definition the coefficients $c(\ell,k) \to 0$ as $\ell|k| \to 0$.  
    By Lemma \ref{lemma: large scale wave-number sum limit} it follows that 
    \begin{align*}
        \lim_{\ell_I \to \infty}\limsup_{\nu \to 0}\sup_{\ell \in (\ell_I, \tilde{\ell_\nu})} \frac{4}{\ell^d}\int_0^\ell r^{d-1}a_{vel}(r)\;dr &=
        \frac{2\beta_d(1)}{d+2}\sum_j\|f_j\|_{L^2}^2 -\frac{2\beta_d(1)}{d+2}\sum_j(\|f_j\|_{L^2}^2 - |\widehat{f_j}(0)|^2)\\
        &= \frac{2\beta_d(1)}{d+2}\sum_j|{f_j}(0)|^2 = 0
    \end{align*}
    as the $f_j$ are all mean zero.
    
    \textbf{Step 2:} Next we show that 
    \[
        \frac{-4\nu}{\ell}(\Gamma_{vel}^\parallel)'(\ell) = 4\beta_d(2)\varepsilon - \sum_{k}c(\ell,k)\E|\widehat{\omega}(k)|^2
    \]
    where the coefficients satisfy $\ds\lim_{\ell|k| \to 0}c(\ell,k) = 0$ and $\ds\lim_{\ell|k| \to \infty}c(\ell,k) = -4\beta_d(2)$ (Note these coefficients $c(\ell,k)$ may be different from the ones found in Step 1). 

    From Proposition \ref{prop: KHM relatons and regularity}, $\Gamma_{vel}^\parallel \in C^2$. Thus we Taylor expand $(\Gamma_{vel}^\parallel)'$ about $r = 0$ and use the Lagrange formulation of the error to get
    \[
        \frac{-4\nu}{\ell}(\Gamma_{vel}^\parallel)'(\ell) =  \frac{-4\nu}{\ell}(\Gamma_{vel}^\parallel)'(0) - \frac{4\nu}{\ell}\int_0^\ell (\Gamma_{vel}^\parallel)''(r)\;dr
    \]
    As $\fint_{S^{d-1}} n_in_jn_k\;dS(n) = 0$ it follows from Fubini's theorem that 
    \begin{align*}
        \frac{-4\nu}{\ell}(\Gamma_{vel}^\parallel)'(0) &= \frac{4\nu}{\ell}\E\fint_{S^{d-1}}\fint_{\T_\lambda^d}n_in_jn_k\partial_ku_iu_j\;dxdS(n) = 0
    \end{align*}
    Next, consider the error term, which we will note is real valued. Using Fourier Analysis, Fubini's Theorem, and Lemma \ref{lemma: longitudinal term integral} when $p = 1$ it follows that 
    \begin{align*}
        \frac{-4\nu}{\ell}\int_0^\ell (\Gamma_{vel}^\parallel)''(r)\;dr &= \sum_k\frac{4\nu}{\ell}\int_0^\ell \fint_{S^{d-1}} n_in_j(n \cdot k)^2e^{-ir(n\cdot k)}\E\fint_0^T\widehat{u_i}(k)\widehat{u_j}(k)\;dtdS(n)dr\\
        &= \sum_k\frac{4\nu}{\ell}\int_0^\ell\sum_{m=0}^\infty \beta_d(m+2)\frac{(-1)^m(2m+2)!(r|k|)^{2m}}{2^{m+1}(m+1)!(2m)!}\;dr\E\fint_0^T|\widehat{\grad u^\nu}(k)|^2\\
        &= 4\nu\sum_k\sum_{m=0}^\infty \beta_d(m+2)\frac{(-1)^m(\ell|k|)^{2m}}{2^{m}m!}\E\fint_0^T|\widehat{\grad u^\nu}(k)|^2\\
        &= 4\nu\beta_d(2)\E\fint_0^T\|\grad u^\nu\|^2 + 4\nu\sum_k\sum_{m=1}^\infty \beta_d(m+2)\frac{(-1)^{m}(\ell|k|)^{2m}}{2^mm!}\E\fint_0^T|\widehat{\grad u^\nu}(k)|^2\\
         &=: 4\beta_d(2)\varepsilon - 4\beta_d(2)\E D(u^\nu) + \sum_kc(\ell,k)\E\fint_0^T|\widehat{\grad u^\nu}|^2
    \end{align*}
    Note that by construction $\ds\lim_{\ell|k| \to 0}c(\ell,k) = 0$. Moreover, by a similar argument to that done in step 1, one can show that $\ds\lim_{\ell|k| \to \infty}c(\ell,k) = -4\beta_d(2)$

    \textbf{Step 3:} Apply Lemma \ref{lemma: large scale wave-number sum limit} to the result from step 2 to conclude that there exists $\tilde{\ell}_\nu \to \infty$ such that 
    \begin{align*}
         \lim_{\ell_I \to \infty}\limsup_{\nu \to 0}\sup_{\ell \in (\ell_I, \tilde{\ell}_\nu)}\frac{-4\nu}{\ell}(\Gamma_{vel}^\parallel)'(\ell) &= 4\beta_d(2)\varepsilon + \lim_{\ell_I \to \infty}\limsup_{\nu \to 0}\sup_{\ell \in (\ell_I, \tilde{\ell}_\nu)}\sum_kc(\ell,k)\E|\widehat{\grad u^\nu}|^2\\
         &= 4\beta_d(2)\varepsilon - 4\beta_d(2)\E D(u^\nu) - 4\beta_d(2)\big(\varepsilon - \E D(u^\nu) - \varepsilon^* \big)\\
         &= 4\beta_d(2)\varepsilon^*
    \end{align*}
    if and only if there exists $M_\nu \to 0$ such that 
    \[
        \liminf_{\nu \to 0}\nu\sum_{|k| \leq M_\nu}\E|\widehat{\grad u^\nu}|^2 = \varepsilon^*
    \]
    \textbf{Step 4:} Next we use Equation \eqref{eq: S0 limit inverse}, which was proved in Section \ref{sec: proof of S0 inverse}, to conclude that the if and only if established in step 3, also equivalently shows that
    \[
        \frac{2}{\ell^{d+2}}\int_0^\ell r^dS_{vel}(r)\;dr = \frac{2}{\ell^{d+2}}\int_0^\ell r^{d+1}\frac{S_{vel}(r)}{r}\;dr = \frac{8\beta_d(1)}{d+2}\varepsilon^* + o_{\ell \to 0}(1)
    \]

    Finally combine Steps 1-4 with the KHM equation to derive \eqref{eq: S long limit inverse}.

\section{Appendix}\label{sec: appendix energy balance}
In this section we provide a proof of Theorem \ref{thm: energy balance}. This is a simple application of the ideas from both \cite{flandoli1995martingale} and \cite{duchon2000inertial}, however the author could find no explicit mention to this fact so it is included here for reference. \newline

First, the existence of statistically stationary solutions to \eqref{eq: incompressible NSE} in the sense of Definition \ref{def: stationary martingale sol} follows from adapting the argument in \cite{flandoli1995martingale} from a domain with Dirichlet boundary conditions to the torus. The details are omitted as the adjustments are trivial. \newline 

First, we rewrite the identity \eqref{eq: weak energy identity} in differential form:
for all $\psi \in C_c^\infty$ and $t \geq 0$
\[
    \int_{\T^d} \psi \cdot du^\nu\;dx + \int_{\T^d} \psi \cdot (u^\nu \cdot \grad) u^\nu\;dxdt = \nu\int_{\T^d}\grad u^\nu :\grad \psi\;dxdt + \sum_{j}\int_{\T^d}f_j\cdot \psi\;dxdW_t^j \quad\quad \P-a.s. 
\]
Let $\phi \in C_c^\infty$ such that $\phi \geq 0$ and $\int \phi \;dx = 1$. For $\gamma \ll 1$ we define $\phi_\gamma(x) = \gamma^{-d}\phi(x/\gamma)$ and mollify
\[
    u_\gamma^\nu = \phi_\gamma *u^\nu, \quad (u_i^\nu u_j^\nu)_\gamma = \phi_\gamma *(u_i^\nu u_j^\nu), \quad p_\gamma^\nu = \phi_\gamma^\nu *p, \quad (f_j)_\gamma = \phi_\gamma *f_j.
\]
Each of these mollified variables now resides in $C_c^\infty$, and as $u$ solves the Navier-Stokes Equations \eqref{eq: incompressible NSE} it follows that:
\begin{align*}
    du_\gamma^\nu + (\grad \cdot (u^\nu\otimes u^\nu)_\gamma + \grad p_\gamma^\nu)dt &= \nu\Delta u_\gamma^\nu\;dt + \sum_{j} (f_j)_\gamma \;dW_t^j\\
    \grad \cdot u_\gamma^\nu &= 0
\end{align*}

Fix $t$, we apply Ito's lemma, integration by parts, and the independence of the Brownian motions to get
\begin{align*}
    d\int_{\T^d}(u^\nu \cdot u_\gamma^\nu)\;dx &= \int_{\T^d}u_\gamma^\nu \cdot du^\nu\;dx + \int_{\T^d}u^\nu \cdot du_\gamma^\nu\;dx + \int_{\T^d}d[u^\nu, u_\gamma^\nu]\;dx\\
    &= -2\nu\int_{\T^d}\grad u^\nu :\grad u_\gamma^\nu\;dxdt +\int_{\T^d}\Big( (u^\nu \otimes u^\nu): \grad u_\gamma^\nu - u^\nu\cdot \Div (u^\nu \otimes u^\nu)_\gamma\Big)\;dxdt\\
    &\quad + \sum_j\sum_k \int_{\T^d}f_j(f_k)_\gamma d[W_t^j, W_t^k]dx + \sum_{j}\int_{\T^d}\Big(f_j\cdot u_\gamma^\nu + (f_j)_\gamma\cdot u^\nu\Big)\;dxdW_t^j\\ 
    &= -2\nu\int_{\T^d} \grad u^\nu :\grad u_\gamma^\nu \;dxdt + \sum_j \int_{\T^d} f_j(f_j)_\gamma \;dxdt - A_\gamma \;dt\\
    &\quad + \sum_j\int_{\T^d} (f_j\cdot u_\gamma^\nu + (f_j)_\gamma u^\nu)\;dxdW_t
\end{align*}
where 
\[
    A_\gamma(s) = \int_{\T^d} u^\nu(s)\cdot \Div(u^\nu(s) \otimes u^\nu(s))_\gamma\;dx - \int_{\T^d} (u^\nu(s) \otimes u^\nu(s)): \grad u_\gamma^\nu(s)\;dx
\]
Or in integral form:
\begin{align*}
    \int_{\T^d}u^\nu(T) \cdot u_\gamma^\nu(T)\;dx &= \int_{\T^d}u^\nu(0) \cdot u_\gamma^\nu(0)\;dx -2\nu\int_0^T\int_{\T^d}\grad u^\nu(t) :\grad u_\gamma^\nu(t) \;dxdt\\
    &\quad + \sum_j \int_0^T\int_{\T^d}f_j(f_j)_\gamma \;dxds - \int_0^T A_\gamma(t) \;dt\\
    &\quad + \sum_j\int_0^T\int_{\T^d} \big(f_j\cdot u_\gamma^\nu(t) + (f_j)_\gamma\cdot u^\nu(t))\;dxdW_t
\end{align*}
As the $f_j$ are smooth, the forcing terms are Lipschitz. Thus when we take the expected value of both sides we get
\begin{align*}
    \E\int_{\T^d}u^\nu(T) \cdot u_\gamma^\nu(T)\;dx &= \E\int_{\T^d}u^\nu(0) \cdot u_\gamma^\nu(0)\;dx -2\nu\E\int_0^T\int_{\T^d}\grad u^\nu(s) :\grad u_\gamma^\nu(t) \;dxdt\\
    &\quad - \E\int_0^TA_\gamma(t) \; dt + \sum_j\int_0^T\int_{\T^d}f_j(f_j)_\gamma \;dxdt
\end{align*}
As $u^\nu \in L_t^2H_x^1$, it follows from the continuity of translations in $L^2$ that for any fixed time $t \geq 0$
\begin{align*}
    \Big|\E\int_{\T^d}u^\nu(t)\cdot u_\gamma^\nu(t)\;dx - \E\int_{\T^d}u^\nu(t)\cdot u^\nu(t)\;dx\Big| &= \Big|\E\int_{\T^d}\phi_\gamma(y) \int_{\T^d}u^\nu(t,x+y)\cdot u^\nu(t,x) - u^\nu(t,x)\cdot u^\nu(t,x)\;dxdy\Big|\\
    &\leq C\E\int_{\supp \;\phi_\gamma}\|u^\nu(t)\|_{L^2}\Big(\int_{\T^d}|u^\nu(t,x+y) - u^\nu(t,x)|^2\;dx\Big)^{1/2}\;dy\\
    &\to 0
\end{align*}
as $\gamma \to 0$. Similar analysis shows that $\int_0^T\int_{\T^d}f_j(f_j)_\gamma \;dxdt \to \int_0^T\|f_j\|_{L^2}^2\;dt$ and $\E\int_0^T\int_{\T^d}\grad u^\nu :\grad u_\gamma^\nu \;dxdt \to \E\int_0^T\|\grad u^\nu(t)\|_{L^2}^2\;dt$ as $\gamma \to 0$. 
Hence by the stationarity of the law of $u^\nu$:
\begin{equation}\label{eq: most of energy balance}
    \lim_{\gamma \to 0}\E\fint_0^T A_\gamma\;dt = -2\nu\E\fint_0^T\|\grad u^\nu\|_{L^2}^2 + \sum_j \|f_j\|_{L^2}^2
\end{equation}

Next, inspired by \cite{duchon2000inertial} we consider the functional: $D_\gamma(u^\nu) := \frac{1}{4}\fint_0^T\int_{\T^d}\int_{\T^d} \grad \phi_\gamma(y) \cdot \delta_yu^\nu|\delta_y u^\nu|^2\;dydxdt$. Let $T_yg(x) = g(x+y)$. Applying the incompressibility condition, Fubini, and integration by parts we get:
\begin{align*}
    \int_{\T^d}\int_{\T^d} \grad \phi_\gamma(y) \cdot \delta_yu^\nu|\delta_y u^\nu|^2\;dydx &= \int\int \partial_{y_j}\phi_\gamma(y)\Big(T_yu_j^\nu T_yu_k^\nu T_yu_k^\nu - 2T_yu_j^\nu T_yu_k^\nu u_k^\nu\\
    &\quad + T_yu_j^\nu u_k^\nu u_k^\nu - u_j^\nu T_yu_k^\nu T_yu_k^\nu + 2u_j^\nu u_k^\nu T_yu_k^\nu - u_j^\nu u_k^\nu u_k^\nu\Big)\;dydx\\
    &= \int\partial_{y_j}\phi_\gamma(y)\Big(\int u_j^\nu u_k^\nu u_k^\nu\;dx - 2\int T_yu_j^\nu T_yu_k^\nu u_k^\nu\;dx + \int T_yu_j^\nu u_k^\nu u_k^\nu\;dx\\
    &\quad - \int T_{-y}u_j^\nu u_k^\nu u_k^\nu\;dx + 2\int u_j^\nu u_k^\nu T_yu_k^\nu\;dx - \int u_j^\nu u_k^\nu u_k^\nu\;dx\Big)\;dy\\
    &= -2\int\partial_{y_j}\phi_\gamma(y)\Big(\int T_yu_j^\nu T_yu_k^\nu u_k^\nu\;dx - \int u_j^\nu u_k^\nu T_yu_k^\nu\;dx\Big)\;dy\\
    &= -2\int u_j^\nu(z)u_k^\nu(z)\partial_{z_j}\int \phi_\gamma(z-x)u_k^\nu(x)\;dxdz\\
    &\quad + 2\int u_k^\nu(z)\partial_{z_j}\int \phi_\gamma(z-x)u_j^\nu(x)u_k^\nu(x)\;dxdz\\
    &= 2\int u_k^\nu\partial_{j}(u_j^\nu u_k^\nu)_\gamma\;dz - 2\int u_j^\nu u_k^\nu \partial_{j} (u_k^\nu)_\gamma \;dz\\
    &= 2A_\gamma
\end{align*}
Define 
\[
    D(u^\nu) := \frac{1}{|\T_\lambda^d|}\lim_{\gamma \to 0}D_\gamma(u^\nu) = \frac{1}{2|\T_\lambda^d|}\lim_{\gamma \to 0}\fint_0^T A_\gamma\;dt
\]
Then after combining $D(u^\nu)$ with Equation \eqref{eq: most of energy balance} the desired energy balance appears:
\[
    \nu\E\fint_0^T\|\grad u^\nu\|_\lambda^2 + \E D(u^\nu) = \frac{1}{2}\sum_j\|f_j\|_\lambda^2
\]
\section{Acknowledgements}
The author would like to thank Dr. Jacob Bedrossian, for suggesting the problem addressed in this paper and providing guidance, as well as Dr. Konstantina Trivisa for her support during the drafting of this article. 

\bibliographystyle{phys}
\bibliography{Ref.bib}
\end{document}